\documentclass[a4paper]{amsart}

\usepackage{amssymb}
\usepackage{amsthm}
\usepackage{amsmath}
\usepackage{mathtools}
\usepackage{amssymb}
\usepackage{latexsym}
\usepackage{graphics}
\usepackage{graphicx}
\usepackage{stackrel}
\usepackage{gensymb}
\usepackage{color}
\usepackage[dvipsnames]{xcolor}
\usepackage{import}  
\usepackage{overpic}
\usepackage{cite}
\usepackage{contour}
\usepackage{verbatim}
\usepackage[shortlabels]{enumitem}

\usepackage[hidelinks,pdftex]{hyperref}

\usepackage[margin=1.2in]{geometry}

\AtBeginDocument{%
   \def\MR#1{}
}

\numberwithin{equation}{section}
\theoremstyle{plain}
\newtheorem{theorem}[equation]{Theorem}
\newtheorem{lemma}[equation]{Lemma}
\newtheorem{proposition}[equation]{Proposition}
\newtheorem{corollary}[equation]{Corollary}
\newtheorem{conjecture}[equation]{Conjecture}

\newtheorem*{namedtheorem}{\theoremname}
\newcommand{\theoremname}{testing}
\newenvironment{named}[1]{\renewcommand{\theoremname}{#1}\begin{namedtheorem}}{\end{namedtheorem}}

\theoremstyle{definition}
\newtheorem{definition}[equation]{Definition}

\newtheorem{remark}[equation]{Remark}

\newcommand{\HH}{{\mathbb{H}}}
\newcommand{\RR}{{\mathbb{R}}}
\newcommand{\ZZ}{{\mathbb{Z}}}

\newcommand{\CC}{{\mathbb{C}}}

\newcommand{\vol}{\operatorname{Vol}}

\newcommand{\bdy}{\partial}
\renewcommand{\setminus}{{\smallsetminus}}

\newcommand{\qroot}{e^{(2\pi\sqrt{-1})/r}}
\newcommand{\Tet}{{\mathrm{Tet}}}
\newcommand{\sgn}{\operatorname{sgn}}

\newcommand{\refthm}[1]{Theorem~\ref{Thm:#1}}
\newcommand{\reflem}[1]{Lemma~\ref{Lem:#1}}
\newcommand{\refprop}[1]{Proposition~\ref{Prop:#1}}
\newcommand{\refcor}[1]{Corollary~\ref{Cor:#1}}

\newcommand{\refconj}[1]{Conjecture~\ref{Conj:#1}}
\newcommand{\refeqn}[1]{\eqref{Eqn:#1}}

\newcommand{\refdef}[1]{Definition~\ref{Def:#1}}
\newcommand{\refsec}[1]{Section~\ref{Sec:#1}}
\newcommand{\reffig}[1]{Figure~\ref{Fig:#1}}

\begin{document}

\title[Augmented links, shadow links, and the TV volume conjecture]{Augmented links, shadow links, and the TV volume conjecture: A geometric perspective}

\author{Dionne Ibarra}
\address{School of Mathematics, Monash University, Clayton, VIC 3800, Australia}
\email{dionne.ibarra@monash.edu}

\author{Emma N.~McQuire}
\address{School of Mathematics, Monash University, Clayton, VIC 3800, Australia}
\email{emma.mcquire@monash.edu}

\author{Jessica S.~Purcell}
\address{School of Mathematics, Monash University, Clayton, VIC 3800, Australia}
\email{jessica.purcell@monash.edu}

\subjclass[2020]{57K10, 57K12, 57K16, 57K31, 57K32}

\begin{abstract}
For hyperbolic 3-manifolds, the growth rate of their Turaev--Viro invariants, evaluated at a certain root of unity, is conjectured to give the hyperbolic volume of the manifold. This has been verified for a handful of examples and several infinite families of link complements, including fundamental shadow links. Fundamental shadow links lie in connected sums of copies of $S^1\times S^2$, and their complements are built of regular ideal octahedra.
Another well-known family of links with complements built of regular ideal octahedra are the octahedral fully augmented links in $S^3$. The complements of these links are now known to be homeomorphic to complements of fundamental shadow links, using topological techniques.
In this paper, we give a new, geometric proof that complements of octahedral fully augmented links are isometric to complements of fundamental shadow links. We then use skein theoretic techniques to determine formulae for coloured Jones polynomials of these links. In the case of no half-twists, this gives a new, more geometric verification of the Turaev--Viro volume conjecture for these links.
\end{abstract}
\maketitle

\section{Introduction}

A major initiative in quantum and geometric topology is to relate geometric invariants of knots and 3-manifolds to quantum invariants. An example of a geometric invariant is the hyperbolic volume, which is unique for any hyperbolic 3-manifold with torus boundary, such as a knot complement, by Mostow--Prasad rigidity~\cite{MostowRigidity, PrasadRigidity}. An example of a quantum invariant is the Jones polynomial~\cite{Jones1985}; other examples include coloured Jones polynomials \cite{ReshetikhinTuraevRibbonGraphInvars} and Turaev--Viro invariants~\cite{TuraevViro}, defined using quantum groups and $3$-manifold triangluations respectively. Volume conjectures relate these, stating that the growth rate of certain quantum invariants are given by the volume of the corresponding knot complement. In this paper, we study a volume conjecture by Chen and Yang~\cite{ChenYang}, concerning the Turaev--Viro invariants. 

The Turaev--Viro invariant $TV_r(M,q)$ of a compact $3$-manifold $M$ is a real-valued invariant, depending on an integer $r\ge3$, given by a Laurent polynomial in variable $q$, where $q$ is a $2r^{th}$ or $r^{th}$ root of unity. The invariant was formulated in the early 1990s by Turaev and Viro~\cite{TuraevViro}. In 2018, Chen and Yang~\cite{ChenYang} extensively calculated growth rates of these invariants for the case $q=\qroot$ and stated the following conjecture.

\begin{conjecture}[TV Volume Conjecture~\cite{ChenYang}]\label{Conj:TVconj}
Let $M$ be a hyperbolic 3-manifold, either closed, cusped or with totally geodesic boundary and volume $\vol(M)$. As $r$ varies over the odd integers, and $q=\qroot$,
\[ \lim_{r\to\infty}\frac{2\pi}{r}\log\left|TV_r(M,q)\right|=\vol(M). \]
\end{conjecture}

\refconj{TVconj} has been verified for the complements of the Borromean rings and the figure-8 knot by Detcherry, Kalfagianni and Yang~\cite{DKYRTTV}, for hyperbolic manifolds obtained by integral and rational Dehn surgery on the figure-8 knot by Ohtsuki~\cite{Ohtsuki} and Wong and Yang~\cite{WongYangFig8DehnFill} respectively, and for certain octahedral links in $S^3$ called Whitehead chains by Wong~\cite{WongWhitehead}. In~\cite{Kumar}, Kumar constructs infinite families of links in $S^3$ which satisfy \refconj{TVconj}.
An extension of the conjecture to Gromov norms has been proposed by Detcherry and Kalfagianni~\cite{DKGromovNorm} and proved for several examples~\cite{KumarMelbyTVadditivity,DKGromovNorm,DKYRTTV}. \refconj{TVconj} has also been shown to be stable under certain link cabling and satellite operations~\cite{DKGromovNorm,DetcherryTVCabling, KumarMelby:Cabling}.

In~\cite{MainPaper}, Belletti, Detcherry, Kalfagianni and Yang proved \refconj{TVconj} for an infinite class of links in connected sums $\#^{c+1} (S^1\times S^2)$ for any $c>0$, known as fundamental shadow links. These links were first considered by Costantino and Thurston~\cite{CostThur3Mani}. They have complements decomposing into hyperbolic regular ideal octahedra, and they form a universal class in the sense that any orientable 3-manifold with empty or toroidal boundary can be obtained from the complement of a fundamental shadow link by Dehn filling~\cite{CostThur3Mani}.

Recently, Wong and Yang showed using surgery descriptions that certain fundamental shadow links coincide with a subfamily of other links with well-known geometric properties, called fully augmented links~\cite[Proposition~6.2]{WongYang}. These are links in $S^3$, first studied by Adams~\cite{Adams} and Agol and D.~Thurston~\cite[Appendix]{LackenbyAltLinks}. Wong and Yang showed that all fully augmented links with complements decomposing into regular ideal octahedra have complements homeomorphic to fundamental shadow links, reframing previous work of Van der Veen~\cite{VanDerVeen:FALvolConj}. Kumar also showed a correspondence between some of these links in~\cite{Kumar}. 
Thus, by the work of~\cite{MainPaper,WongYang}, \refconj{TVconj} is true for all octahedral fully augmented links.

The proof in~\cite{WongYang} is topological and combinatorial. However, fully augmented links are highly geometric. In this paper, we give a new proof of the correspondence of these link complements, exploiting their geometry. 

\begin{named}{\refthm{octareFSL}}
Let $L$ be an octahedral fully augmented link with $c$ crossing circles. Then $S^3-L$ is isometric to $\#^c (S^1\times S^2)-\Tilde{L}$ for some fundamental shadow link $\Tilde{L}$.
\end{named}

The proof of \refthm{octareFSL} draws on the circle packing structure that can be obtained from the polyhedral decomposition of fully augmented links described by Agol and Thurston~\cite[Appendix]{LackenbyAltLinks}. Our proof complements the proofs of~\cite{WongYang, VanDerVeen:FALvolConj}, in that we use explicit geometric properties of the link to show the correspondence. This also highlights the geometry of fundamental shadow links. Our methods further give a correspondence between the diagrams produced by Wong and Yang and the circle packing geometry of fully augmented links found in the literature; see for example~\cite{PurcellFullyAugLinks, PurcellCuspShapes, JessicaBook}.

We next give a diagrammatic construction of the coloured Jones polynomial for octahedral fully augmented links, using recoupling theory of~\cite{KauffSpinNetworks,KauffLins,LickorishTextbook}. Our proof has similarities to work of Van der Veen~\cite{VanDerVeen:FALvolConj}, but we consider slightly different links, a different colouring, a different root of unity, and his focus is a different volume conjecture.
When computing the coloured Jones polynomial, we observe that the structure of the circle packing is closely matched by the quantum contributions, in that adding one octahedron to one of the polyhedra in the polyhedral decomposition contributes exactly one quantum $6j$-symbol in the recoupling theory. Quantum $6j$-symbols are known to have growth rates that are maximally given by the volume of a regular ideal hyperbolic octahedron~\cite{MainPaper}. We then study the Turaev--Viro invariants of fully augmented links by using a result of Detcherry, Kalfagianni and Yang \cite{DKYRTTV}, which gives the Turaev--Viro invariants as a sum of coloured Jones polynomials.

This leads to a diagrammatic proof of \refconj{TVconj} for all octahedral fully augmented links without half-twists. Note the result also follows from~\cite{WongYang} and~\cite{MainPaper}, but our proof uses different methods. Our techniques are only available for diagrams of links in the 3-sphere, and not for links in other manifolds such as connect sums of $S^1\times S^2$. 

\begin{named}{\refthm{TVconjforOctFAL}}
Let $L$ be an octahedral fully augmented link with $c$ crossing circles and without half-twists. Then as $r$ varies over the odd integers,
\[\lim_{r\to\infty}\frac{2\pi}{r}\log|TV_r(S^3\setminus L,q=\qroot)|=2(c-1)v_8=\vol(S^3\setminus L),\]
where $v_8\cong 3.66$ is the volume of a regular ideal hyperbolic octahedron.
\end{named}

We note related work on a different volume conjecture, namely the volume conjecture of Kashaev~\cite{Kashaev}, reformulated by Murakami, Murakami~\cite{MurakamiMurakami}. That volume conjecture relates the coloured Jones polynomials to the hyperbolic volume of hyperbolic knot complements, evaluated at $q=A^2=e^{(\pi\sqrt{-1})/r}$. Van der Veen proved the Kashaev--Murakami--Murakami volume conjecture holds for Whitehead chains and knotted trivalent graphs~\cite{VanDerVeen:WhiteheadChains, VanDerVeen:FALvolConj}, which include octahedral fully augmented links. Again our work is related, but using different quantum invariants at different roots of unity.

An outline of the paper is as follows. In \refsec{FALandFSL}, we introduce fully augmented links and fundamental shadow links, and give our geometric proof that fully augmented links correspond to fundamental shadow links.
In \refsec{CJPforOFAL}, we give a new calculation of the coloured Jones polynomials of octahedral fully augmented links, using combinatorics of the diagram and recoupling theory. Finally in \refsec{TV}, we give a new proof of the TV Volume Conjecture for octahedral fully augmented links without half-twists.

\section{Fully augmented links and shadow links}\label{Sec:FALandFSL}

In this section, we give a geometric proof that octahedral fully augmented links are instances of fundamental shadow links. We begin by reviewing constructions. 

\subsection{Fully Augmented Links} \label{Sec:FAL}
We introduce fully augmented links and describe their key properties, following the exposition of~\cite{PurcellFullyAugLinks}; see also~\cite[Chapter~7]{JessicaBook}.

\begin{definition}
A \emph{flat fully augmented link} is a link in $S^3$ consisting of two types of components:
\begin{enumerate}
\item Components lying embedded on the plane of projection, called \emph{knot strands};
\item Unknotted components, perpendicular to the plane of projection, called \emph{crossing circles}. Each crossing circle bounds a disc punctured by exactly two knot strands; this is called a \emph{crossing disc}.
\end{enumerate}
A \emph{fully augmented link} is obtained from a flat fully augmented link by adding single crossings, or \emph{half-twists}, between two knot strands running through a crossing disc, for some (possibly empty) subset of the crossing discs.
A fully augmented link is \emph{reduced} if its diagram is connected, nonsplit, prime, and there are no parallel crossing circles.
\end{definition}

\begin{remark}
In the literature, fully augmented links are often described as links obtained by augmenting prime, twist-reduced knot diagrams, as in~\cite{PurcellFullyAugLinks}, since there are applications to the geometric properties of such knots after Dehn filling. Here, we do not consider Dehn fillings. Our definition coincides with that of~\cite{PurcellFullyAugLinks} after removing pairs of crossings from the augmented diagrams. 
\end{remark}

Fully augmented links have a polyhedral decomposition of their complements, described by Agol and Thurston~\cite[Appendix]{LackenbyAltLinks}. We summarise in the following proposition. Although the proof appears in~\cite{LackenbyAltLinks, PurcellFullyAugLinks, JessicaBook}, we also include a summary here. 

\begin{proposition} \label{Prop:FALpolydecomp}
Let $L$ be a fully augmented link. The link complement $S^3\setminus L$ can be decomposed into two identical ideal polyhedra with the following properties:
\begin{enumerate}
\item Faces of the polyhedra are checkerboard coloured. Shaded faces are triangles corresponding to crossing discs, and white faces correspond to components of the projection plane.
\item All ideal vertices are 4-valent.
\item Each edge class in the polyhedral gluing of the link complement contains exactly 4 edges. 
\end{enumerate}
\end{proposition}

\begin{proof}
First, cut along the projection plane to slice the complement into two identical pieces. For each piece, slice each remnant of a crossing disc up the middle, to replace it with two parallel copies. If the crossing disc is adjacent to a half-twist, unwind the half-twist by reflecting the crossing disc in a vertical axis. Then collapse each remnant of the link into a single ideal vertex. This gives two identical polyhedra, one from each piece coming from slicing the projection plane in half. This process is shown in Figure~\ref{Fig:polydecomp}. Note that each shaded face is adjacent to exactly one crossing circle. Observe also that ideal vertices meet two shaded and two white faces, hence are 4-valent. Finally, observe that any edge lies in the intersection of white and shaded faces; four of these meet in the fully augmented link. 

\begin{figure}
\centering
\includegraphics{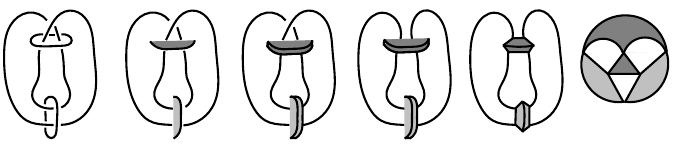}
\caption{The steps to obtain a polyhedral decomposition of the link complement of a fully augmented link.}
\label{Fig:polydecomp}
\end{figure}

One can obtain $S^3\setminus L$ from the two polyhedra by reversing the decomposition procedure as follows. Shaded triangles are glued to shaded faces on the same polyhedron in pairs by folding across a vertex corresponding to a crossing circle. If there are half-twists, the gluing is modified slightly; shaded triangles are still glued to shaded faces on the same polyhedron in pairs across a vertex corresponding to a crossing circle, but the gluing is by reflection fixing the ideal vertex coming from the crossing circle. Finally, glue corresponding white faces of the two polyhedra. That is, reflect across the white faces. 
\end{proof}

Note that our gluing for half-twists varies slightly from \cite{PurcellFullyAugLinks}, but gives an equivalent 3-manifold. 

In~\cite[Theorem~6.1]{PurcellCuspShapes}, it is shown that the complement of a reduced fully augmented link with at least two crossing circles admits a complete hyperbolic structure. This is proven by showing that the ideal polyhedra obtained in \refprop{FALpolydecomp} correspond to totally geodesic hyperbolic polyhedra, which glue to give the link complement. The totally geodesic polyhedra are obtained using circle packings, as follows. 

\begin{definition}
A \emph{circle packing} is a finite collection of Euclidean circles either in $\mathbb{R}^2$ or $S^2$ which meet only in points of tangency. The \emph{nerve} of a circle packing is the graph obtained by placing a single vertex in each circle and an edge between two vertices whenever the corresponding circles are tangent.
\end{definition}

\begin{lemma}[Lemma~2.3 of \cite{PurcellFullyAugLinks}] \label{Lem:FALtocirclepacking}
Given a hyperbolic fully augmented link $L$, the polyhedral decomposition of $S^3\setminus L$ corresponds to a circle packing of $S^2$ whose nerve gives a triangulation of $S^2$. The nerve has the following properties:
\begin{itemize}
\item Each edge of the nerve has distinct endpoints.
\item No two vertices are joined by more than one edge.
\end{itemize}
\end{lemma} 

The circle packing in \reflem{FALtocirclepacking} is obtained by taking a circle for each white face of the ideal polyhedral decomposition. The nerve therefore consists of vertices for each white face, with edges corresponding to ideal vertices of the ideal polyhedron. Since edges then enclose shaded faces, which are ideal triangles, the nerve is a triangulation of $S^2$.
View the boundary at infinity of $\mathbb{H}^3$ as the sphere $S^2_{\infty}$. Each circle in the circle packing gives a geodesic plane in $\HH^3$. The interstitial regions between circles form curvilinear triangles corresponding to shaded faces. There is a unique circle through the three ideal vertices of an interstial circle; call this a shaded circle in $\HH^3$. The region bounded by hemispheres coming from white circles and shaded circles gives a totally geodesic polyhedron. 
A circle packing and its nerve is shown in Figure~\ref{Fig:circlepack}.

\begin{figure}
\centering
\includegraphics{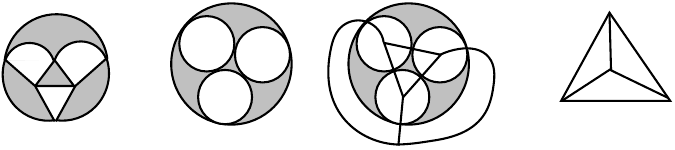}
\caption{Polyhedron (leftmost), circle packing obtained from the polyhedron (second from left), circle packing with its nerve superimposed (second from right), nerve (rightmost).}
    \label{Fig:circlepack}
\end{figure}

The next lemma is a converse to \reflem{FALtocirclepacking}.
\begin{lemma}[Lemma~2.4 of \cite{PurcellFullyAugLinks}]\label{Lem:graphtoFAL}
Let $\gamma$ be a triangulation of $S^2$ such that no two vertices are joined by more than one edge and each edge has distinct ends. Choose a collection of ``red'' edges such that each triangle of $\gamma$ meets exactly one red edge. Then there is a hyperbolic fully augmented link associated to this painted graph; it has $\gamma$ as its nerve.
\end{lemma}

\begin{remark}
The collection of red edges in \reflem{graphtoFAL} is called a \textit{dimer}.
\end{remark}

Given a triangulation of $S^2$ as described in \reflem{graphtoFAL} with a dimer, one constructs a fully augmented link as follows. On each edge in the dimer, place a crossing circle parallel to that edge. Draw two short strands through each crossing circle. Then join two strands together across an edge which is not in the dimer, such that each strand is joined to the nearest strand from another crossing circle. An example is shown in Figure~\ref{Fig:graphFAL}. Note that at each edge in the dimer, one may insert a half-twist or not.
\begin{figure}
\centering
\includegraphics{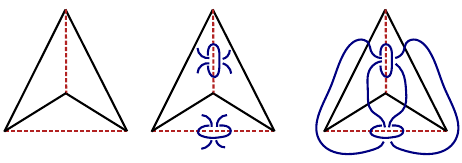}
\caption{Given a triangulation of $S^2$ and a choice of red edges (shown dashed) as in \reflem{graphtoFAL}, one can obtain a fully augmented link as shown.}
    \label{Fig:graphFAL}
\end{figure}

The \textit{dual graph} of a triangulation is obtained by taking a vertex for each face and putting an edge between two vertices whenever the corresponding faces are adjacent across an edge. Given a triangulation $G$ as described in \reflem{graphtoFAL}, one can obtain the corresponding fully augmented link using the dual graph $G'$ of $G$ as follows. First, the dimer of $G'$ is given by the edges of $G'$ which intersect the dimer of $G$. For each edge in the dimer of $G'$, replace the edge by two strands running parallel to the edge and a crossing circle perpendicular to the edge, such that the two strands pass through the crossing circle. Connect the strands following along the non-dimer edges of the graph as shown in Figure~\ref{Fig:dual}. This gives a fully augmented link. Again, at each edge in the dimer there is the choice of whether or not to insert a half-twist.
\begin{figure}
    \centering
    \includegraphics{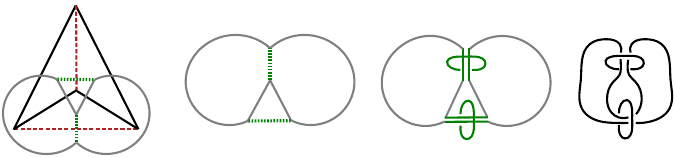}
    \caption{An example of obtaining a fully augmented link from the dual graph of a triangulation. The leftmost diagram shows the triangulation (black) and its dimer (red dashed), the dual graph (grey) and its dimer (green dotted). The following three diagrams show the process of obtaining the fully augmented link starting with the dual graph.}
    \label{Fig:dual}
\end{figure}

\begin{definition}
A fully augmented link with polyhedral decomposition obtained by gluing regular ideal octahedra is called \textit{octahedral}.
\end{definition}

\subsection{Fundamental Shadow Links} 

Fundamental shadow links are a family of links in connected sums of copies of $S^1\times S^2$, first considered by Costantino and Thurston \cite{CostThur3Mani}, who defined them using Turaev's theory of shadows \cite{TuraevShadows}. 
The building blocks of fundamental shadow links are a collection of 3-balls, each with four shaded discs on its boundary and six arcs connecting them, as in Figure~\ref{Fig:block}. Take $c$ of these building blocks and glue them along the shaded discs, in such a way that the endpoint of each arc is glued to the endpoint of another arc, possibly the same arc. This yields a genus $c+1$ handlebody $H_L$ (possibly non-orientable) with a link in its boundary, coming from the arcs. Take the orientable double of this handlebody, that is, glue $H_L$ to an identical copy of itself by the identity map on $\bdy H_L$. 
We obtain a manifold $M_c:=\#^{c+1}(S^1\times S^2)$ with a link $L$ inside. The link $L$ is called a \emph{fundamental shadow link}. The number $c$ is called its \emph{shadow complexity}. 

\begin{figure}
\centering
\includegraphics{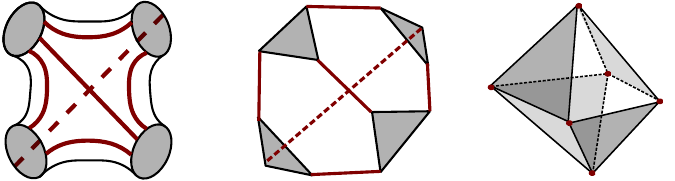}
\caption{Left: Building block of a fundamental shadow link; the six arcs are shown in (thicker) red. Centre: The block is homeomorphic to a truncated tetrahedron. Right: Truncating maximally, removing all edges, gives an ideal octahedron.}
\label{Fig:block}
\end{figure}

\begin{remark}
Note each building block is homeomorphic to a tetrahedron with its vertices truncated.
The shaded discs correspond to triangular faces coming from the truncation; specifying a gluing of building blocks is equivalent to specifying gluing isometries for each pair of triangular faces. There are six possible ways to glue two triangles together, described by the dihedral group on three elements $D_3$. We call the faces of a truncated tetrahedron coming from the truncation \emph{shaded} and the other faces \emph{white}.
\end{remark}

We next describe a relationship between a fundamental shadow link and a 4-valent graph with gluing information. This will specify a fundamental shadow link $L$ up to Dehn twists about meridians on the corresponding handlebody $H_L$. Since such Dehn twists give homeomorphisms of the link complement, our graph relationship will suffice to uniquely determine the link complement. 
Descriptions of fundamental shadow links in terms of 4-valent graphs can be found in the literature~\cite{CostThur3Mani,Cost6jSymb}; see \cite{Kumar} for a summary.

Let $M$ be the complement in $\#^{c+1}(S^1\times S^2)$ of a fundamental shadow link $L$, obtained by gluing $c$ building blocks.
Construct a labelled 4-valent graph from $M$ as follows. Begin with $c$ disconnected vertices, one for each building block, which we view as a truncated tetrahedron. Draw edges between vertices according the gluing of the shaded faces: that is, if two shaded faces are glued together, draw an edge between the vertices corresponding to truncated tetrahedra with those faces. Finally, label each edge with an element of $D_3$ according to the way in which the corresponding shaded faces are glued to obtain a labelled 4-valent graph. 

Conversely, given a 4-valent graph $G$ with $c$ vertices and with its edges labelled by elements of $D_3$, obtain the complement of a fundamental shadow link as follows. To each vertex $v$ of $G$, associate a truncated tetrahedron $T_v$. Associate each shaded face of $T_v$ to a distinct edge incident to $v$. If there is a loop incident to $v$, associate each endpoint of the loop to distinct triangular faces. Then each edge $e$ of $G$ has two shaded triangular faces associated to it. Glue together the two shaded faces associated to $e$, where the gluing is determined by the label on $e$. This gives rise to a genus $c+1$ handlebody. Finally, double this handlebody to obtain $\#^{c+1}(S^1\times S^2)$ with a link inside, coming from the edges of the truncated tetrahedra. The link will be a fundamental shadow link, since its complement is obtained by gluing the building blocks of \reffig{block}.

\begin{definition} \label{Def:gluingGraph}
A \textit{gluing graph} is a 4-valent graph $G$ with edges labelled by elements of $D_3$. The edge labels are called the \emph{gluing information} of $G$.
\end{definition}

Observe that if the vertices of a tetrahedron are truncated maximally, as on the right of \reffig{block}, the result is an octahedron. Every fundamental shadow link can be given a hyperbolic structure by realising the building blocks as regular ideal octahedra, and gluing according to the gluing graph. The following proposition is then an immediate consequence. 

\begin{proposition}[Costantino and Thurston~\cite{CostThur3Mani}]\label{Prop:shadowlinkvolume}
Let $L\subset M_c$ be a fundamental shadow link. Then the complement $M_c\setminus L$ can be given a complete hyperbolic metric with volume $2cv_8$, where $v_8$ is the volume of a regular ideal hyperbolic octahedron, and $c$ is the shadow complexity. \qed
\end{proposition}

\subsection{Correspondence between links}
We now prove \refthm{octareFSL}, that octahedral fully augmented links are fundamental shadow links.

\begin{theorem}\label{Thm:octareFSL} 
Let $L$ be an octahedral fully augmented link with $c$ crossing circles. Then $S^3\setminus L$ is isometric to $\#^c (S^1\times S^2)\setminus \Tilde{L}$ for some fundamental shadow link $\Tilde{L}$.
\end{theorem} 

Our main tool is the following definition and result from \cite{PurcellFullyAugLinks}. 

\begin{definition}
Given a triangle $T$, the \textit{central subdivision} of $T$ is the subdivision obtained by inserting a vertex in the centre of $T$, then adding three edges which run from the new vertex to each of the three vertices of $T$. See Figure~\ref{Fig:adding octa}, left.
\end{definition}

\begin{proposition}[Proposition~3.8 of \cite{PurcellFullyAugLinks}] \label{Prop:octaLinksCharac}
Let $L$ be a fully augmented link with decomposition into two ideal polyhedra isometric to $P$. Let $N$ be the nerve associated with the circle packing of $L$. Then $P$ is obtained by gluing regular ideal octahedra if and only if $N$ is obtained by central subdivision of the complete graph on four vertices. In this case, there are $c-1$ such octahedra, where $c$ is the number of crossing circles in the diagram of $L$.
\end{proposition}

\begin{figure}
\centering
\includegraphics{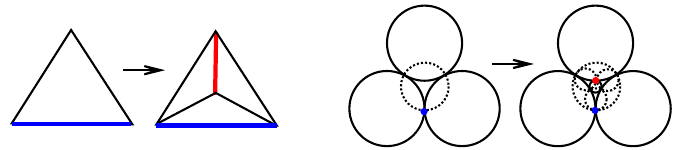}
\caption{Subdividing a triangle in the nerve (left) corresponds to adding a circle to the circle packing (right). A dotted line indicates a shaded face.}
\label{Fig:adding octa}
\end{figure}

Let $L$ be an octahedral fully augmented link with complement decomposing into two polyhedra isometric to $P$. Since $P$ is a union of a finite number of regular ideal octahedra, by \refprop{octaLinksCharac}, the nerve $N$ associated to the circle packing of $L$ is obtained by subdividing the complete graph on four vertices a finite number of times. We will prove \refthm{octareFSL} by induction on the number of times we must subdivide the complete graph on four vertices to obtain $N$. 
If zero times, then $L$ corresponds to the Borromean rings or one of the Borromean twisted sisters; see \reffig{BorrRingsandSisters}.

\begin{figure}
\centering
\includegraphics{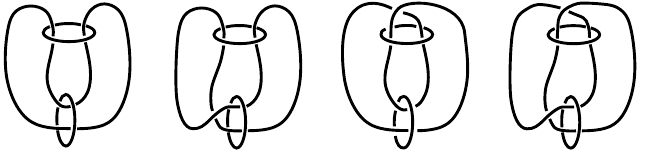}
\caption{Borromean family: Borromean rings (leftmost) and Borromean twisted sisters.}
\label{Fig:BorrRingsandSisters}
\end{figure}

\begin{lemma}\label{Lem:BorrRingsFALtoFSL}
Any fully augmented link of the Borromean family corresponds to a fundamental shadow link with associated graph consisting of a single vertex and two edges that are loops.
\end{lemma}

\begin{proof}
If $L$ is one of the links in the Borromean family, then its ideal polyhedral decomposition consists of two regular ideal octahedra. In particular, $P$ is a single octahedron, as seen from the circle packing of $L$, shown on the left of Figure~\ref{Fig:borr rings}. The colours of the shaded faces indicate the gluing. Applying a M\"obius transformation taking one of the ideal vertices in the leftmost figure to infinity gives the middle figure. We obtain the 4-valent planar graph shown on the right by taking a 4-valent vertex and connecting its edges according to the gluing of the shaded faces of $P$.

\begin{figure}
\centering
\includegraphics{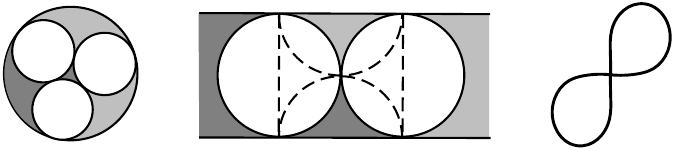}
\caption{The leftmost figure shows the circle packing. Applying a M\"obius transformation taking one of the ideal vertices in the leftmost figure to infinity gives the middle figure. The rightmost figure shows the corresponding 4-valent planar graph.}
\label{Fig:borr rings}
\end{figure}

To turn this 4-valent graph into a gluing graph, view each edge as corresponding to a pair of shaded faces coming from a crossing disc. Label the edges as either the identity isometry, if there is no half-twist adjacent to the corresponding crossing circle, or a reflection if the corresponding crossing circle is adjacent to a half-twist. Here, the reflection is across the ideal vertex of the octahedron shared by the two shaded faces. Thus, when viewing the octahedron as a truncated tetrahedron, the reflection is across the edge that joins the two shaded triangles. Doing so gives a gluing graph $G$. Let $\Tilde{L}$ be the fundamental shadow link obtained from $G$.
Then by construction, $\#^2(S^1\times S^2)\setminus \Tilde{L}$ has identical gluing to $S^3\setminus L$, where we note that reflecting across the white faces in the polyhedral decomposition is equivalent to doubling across the white faces to obtain $\#^2(S^1\times S^2)$.
\end{proof}

\begin{definition}
A \emph{basic graph} is a graph that is a single vertex with two loops, with edges labelled by elements of $D_3$. A \emph{Borromean basic graph} is a basic graph corresponding to fully augmented links in the Borromean family.
\end{definition}

By \reflem{BorrRingsFALtoFSL}, Borromean basic graphs have each edge labelled either by the identity or a reflection giving a half twist.
Note that there are three Borromean basic graphs, corresponding to the Borromean rings or one of the Borromean twisted sisters. Two of the Borromean twisted sisters have homeomorphic complements. Finally, note that both edges in a Borromean basic graph correspond to shaded faces that glue to form crossing discs. When an octahedral fully augmented link corresponds to a fundamental shadow link, and an edge of the associated graph corresponds to gluing two shaded faces into a crossing disc, we will call the edge a \emph{crossing circle edge}. 

\begin{figure}
\centering
 \[\vcenter{\hbox{
 \begin{overpic}{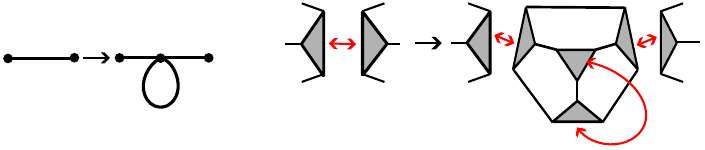}
 \put(0, 14.5){$v_2$}
 \put(5, 10.5){$e$}
 \put(9, 14.5){$v_3$}
 \put(16, 14.5){$v_2$}
 \put(28, 14.5){$v_3$}
 \put(22, 14.5){$v_1$}
 \put(40, 17){$O_2$}
 \put(54, 17){$O_3$}
 \put(63.2, 17){$O_2$}
 \put(80, 17){$O_1$}
 \put(95.5, 17){$O_3$}
\put(43.8, 13.5){$S$}
\put(51.7, 13.5){$S'$}
 \end{overpic}}}\]
\caption{A graph move (left) corresponds to a gluing of tetrahedra shown on the right.
}
\label{Fig:LocalPicFromGluing}
\end{figure}
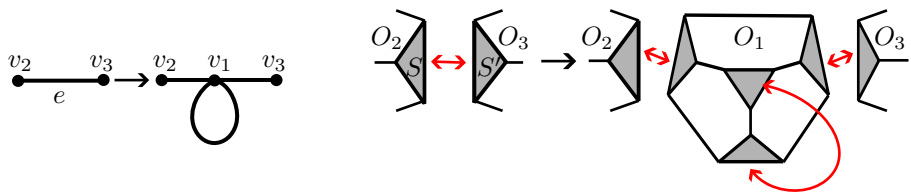

\begin{definition}
Given a gluing graph $G$ as in \refdef{gluingGraph}, define a \emph{graph move} on an edge of $G$ as follows. 
Replace an edge of the graph with three edges and a vertex as illustrated in \reffig{LocalPicFromGluing}~(left). There are two horizontal edges created with vertices labelled $v_2$ and $v_3$.
 Call these the free edges. Label one of the free edges with the identity, and the other labelled the same as the gluing information on the deleted edge. (The two choices of labeling the right versus the left free edge by the identity give isometric links. For concreteness, we will choose to label the free edge on the left by the identity.)
The loop is labelled with either the identity or reflection across the edge joining the two corresponding triangular faces.
We call the new piece added by the graph move the \emph{loop piece.}
\end{definition}

\begin{lemma}\label{Lem:GraphMove}
Suppose the nerve of $L$ differs from that of $L'$ by a central subdivision, where both $L$ and $L'$ are octahedral fully augmented links. Suppose $S^3\setminus L'$ is homeomorphic to the complement of a fundamental shadow link obtained from gluing graph $G'$. Then $S^3\setminus L$ is homeomorphic to the complement of a fundamental shadow link obtained from gluing graph $G$, where $G$ differs from that of $G'$ by a graph move along an edge of $G'$ that is a crossing circle edge. 
\end{lemma}

\begin{proof}
Let $N$ denote the nerve of $L$, and let $N'$ denote the nerve of $L'$, so $N$ is obtained by subdividing $N'$ once. Let $P'$ denote one of the ideal polyhedra in the decomposition of $S^3\setminus L'$. 
Now, subdividing a triangle in the nerve corresponds to adding a circle inside a shaded face of the circle packing~\cite{PurcellFullyAugLinks}, as in Figure~\ref{Fig:adding octa}. Let $O_1$ be the regular ideal octahedron which is attached to $P'$ in the subdivision, so that $O_1$ is attached to some octahedron $O_2$ of $P'$ along a shaded face $S$ of $O_2$. 
The shaded face $S$ is shown with dotted lines before adding the octahedron on the right of Figure~\ref{Fig:adding octa}.

Consider the circle packing. A shaded face corresponds to a crossing disc. So the ideal vertex coloured blue on the right of \reffig{adding octa} comes from a crossing circle. This forces the ideal vertex coloured red to also correspond to a crossing circle. See \reffig{FALlocalPic}. Since shaded faces glue to shaded faces on the same polyhedron across an ideal vertex coming from a crossing circle, this forces the two shaded faces of $O_1$ adjacent to the red vertex to be glued together. 

\begin{figure}
\centering
\includegraphics{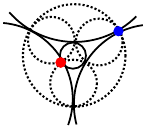}
\hspace{1in}
\includegraphics{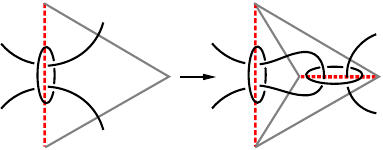}
\caption{Left: Zoomed in diagram of rightmost diagram of Figure~\ref{Fig:adding octa}. The blue and red dots are ideal vertices corresponding to crossing circles. A dotted line indicates a shaded face.
Right: Subdividing the nerve corresponds to altering the link diagram locally as shown.}
\label{Fig:FALlocalPic}
\end{figure}

Now, one of the shaded faces of $O_1$ must be glued to $S$. But in the gluing of $P'$, $S$ is glued to some other shaded face $S'$ of some octahedron $O_3$,  which may or may not be distinct to $O_2$. It follows that the final shaded face of $O_1$ must be glued to $S'$ to complete the gluing. (Note $S$ and $S'$ are distinct since in the gluing of $P'$, each shaded face glues to a distinct shaded face.) Thus, in the gluing of $P$, $O_1$ has one shaded face which glues to $S$, one shaded face which glues to $S'$, and a pair of shaded faces which glue to each other.

By assumption, $L'$ can be viewed as a fundamental shadow link, so there is a corresponding gluing graph $G'$. Let $v_2,v_3$ be the vertices corresponding to $O_2,O_3$ respectively in $G'$, and we may have $v_2=v_3$. Since $S$ glues to $S'$ in the gluing of $P'$, there is an edge $e$ connecting $v_2$ and $v_3$. Let $G$ be the gluing graph encoding the gluing of the octahedra in $P$, then $G$ is obtained from $G'$ by performing a graph move on $e$. To see this, first note that since the gluing of $P$ matches the gluing of $P'$ except for gluings involving $O_1$, $G$ is identical to $G'$ (including gluing information) except for the following. There is a vertex $v_1$ in $G$ corresponding to $O_1$. The edge $e$ between $v_2$ and $v_3$ is not present in $G$; since $O_1$ has shaded faces gluing to $S$ and $S'$, there is an edge from $v_1$ to $v_2$ and an edge from $v_1$ to $v_3$. The remaining two shaded faces of $O_1$ glue to each other, so there is a loop on $v_1$. See Figure~\ref{Fig:LocalPicFromGluing}.

For the gluing information, consider the shaded faces in the octahedral decomposition of $P$. There are two possibilities for the gluing of any two shaded faces: if a pair of shaded face come from a crossing circle, then they are glued by a reflection preserving the crossing circle vertex if the corresponding crossing circle is adjacent to a half-twist, and by the identity isometry otherwise. All other pairs of shaded faces come from splitting $P$ into octahedra, hence are glued by the identity. Label the edges incident to $v_1$ according to the gluing of the corresponding shaded faces: the edge from $v_1$ to $v_2$ is assigned the identity, the edge from $v_1$ to $v_3$  is assigned the edge label of $e$ in $G'$, the loop on $v_1$ is assigned a reflection preserving the vertex on the ideal octahedron corresponding to the crossing circle if there is a half-twist adjacent to the corresponding crossing circle and the identity isometry otherwise. The resulting gluing graph is exactly the gluing graph obtained from $G'$ by performing the graph move on $e$. Note that $e$ is a crossing circle edge; $S,S'$ are shaded faces of $P'$ so they come from a crossing circle disc. 

Let $\Tilde{L}$ be the fundamental shadow link obtained from the gluing graph $G$. By construction, $\Tilde{L}$ and $L$ have homeomorphic complements. Since the gluing is by isometry, the complements are in fact isometric.
\end{proof}

\begin{proof}[Proof of \refthm{octareFSL}]
The proof is by induction on the number of times we perform central subdivision to obtain the nerve of $L$.

If zero times, then $L$ is in the Borromean family, and \reflem{BorrRingsFALtoFSL} implies that $S^3\setminus L$ is a fundamental shadow link with associated graph one of the three Borromean basic graphs. Note both edges on a Borromean basic graph are crossing circle edges. 

So suppose $L$ is obtained by performing central subdivision $k$ times. Let $N'$ be the nerve obtained by performing central subdivision $k-1$ times, with dimer inherited from that of $L$. By induction, $S^3\setminus L'$ is a fundamental shadow link. By \reflem{GraphMove}, $S^3\setminus L$ is also a fundamental shadow link, with associated graph $G$ obtained from that of $L'$ by a graph move along a crossing circle edge. 

Finally, we relate the number of crossing circles to the number of connected sums. If $L$ has $c$ crossing circles, then by \refprop{octaLinksCharac}, $P$ is obtained by gluing $c-1$ regular ideal octahedra, hence there are $c-1$ building blocks. Gluing these gives a genus $c$ handlebody, and doubling yields
$\#^c (S^1\times S^2$).
\end{proof}

\begin{corollary} \label{Cor:octFALandGraphMove}
Let $L$ be a fully augmented link. Then $L$ is octahedral if and only if the gluing graph of $L$ is obtained by graph moves on one of the three Borromean basic graphs. \qed
\end{corollary}

\subsection{Comparison to change-of-pair operation}
In \cite{WongYang}, Wong and Yang describe a surgery move which allows one to transform an octahedral fully augmented link $L$ into a fundamental shadow link $\Tilde{L}$ with homeomorphic complements. Their move is shown in \reffig{WongYangMove}; each crossing circle of $L$ is given the 0-framing and then encircled with a framed trivial unknot. Note that this move is an example of the change-of-pair operation which Wong and Yang define, a topological operation which changes the pair $(M,L)$ where $M$ is a closed oriented 3-manifold and $L\subset M$ is a framed link, but does not change the complement $M\setminus L$. In the case of transitioning from octahedral fully augmented links to fundamental shadow links, the move in \reffig{WongYangMove} can be seen from our proof, specifically, from the loop piece.

\begin{figure}
    \centering
\begingroup%
  \makeatletter%
  \providecommand\color[2][]{%
    \errmessage{(Inkscape) Color is used for the text in Inkscape, but the package 'color.sty' is not loaded}%
    \renewcommand\color[2][]{}%
  }%
  \providecommand\transparent[1]{%
    \errmessage{(Inkscape) Transparency is used (non-zero) for the text in Inkscape, but the package 'transparent.sty' is not loaded}%
    \renewcommand\transparent[1]{}%
  }%
  \providecommand\rotatebox[2]{#2}%
  \newcommand*\fsize{\dimexpr\f@size pt\relax}%
  \newcommand*\lineheight[1]{\fontsize{\fsize}{#1\fsize}\selectfont}%
  \ifx\svgwidth\undefined%
    \setlength{\unitlength}{122.40000343bp}%
    \ifx\svgscale\undefined%
      \relax%
    \else%
      \setlength{\unitlength}{\unitlength * \real{\svgscale}}%
    \fi%
  \else%
    \setlength{\unitlength}{\svgwidth}%
  \fi%
  \global\let\svgwidth\undefined%
  \global\let\svgscale\undefined%
  \makeatother%
  \begin{picture}(1,0.58823528)%
    \lineheight{1}%
    \setlength\tabcolsep{0pt}%
    \put(0,0){\includegraphics[width=\unitlength,page=1]{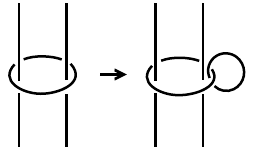}}%
    \put(0.66784129,0.38340213){\color[rgb]{0,0,0}\makebox(0,0)[lt]{\lineheight{1.25}\smash{\begin{tabular}[t]{l}$0$\end{tabular}}}}%
  \end{picture}%
\endgroup%

    \caption{The move described by Wong and Yang.}
    \label{Fig:WongYangMove}
\end{figure}

The following is proved in~\cite{WongYang}. Our methods give a new proof. 

\begin{proposition}
Let $L$ be an octahedral fully augmented link and $\Tilde{L}$ its corresponding fundamental shadow link from \refthm{octareFSL}. Then the diagram of $\Tilde{L}$ is obtained from the diagram of $L$ by applying the change-of-pair move in \reffig{WongYangMove} to each crossing circle.
\end{proposition}

\begin{proof}
Since $L$ is octahedral, by \refcor{octFALandGraphMove}, we prove the statement on the number of graph moves needed to obtain the gluing graph $G$ of $L$ from one of the Borromean basic graphs. Suppose zero moves are needed. Then $L$ is one of the Borromean rings family. Take a single truncated tetrahedron and glue the faces coming from the truncation according to $G$. Since the edges of the tetrahedron form $\Tilde{L}$, one can use the gluing to determine a diagram of $\Tilde{L}$; see \reffig{BorrSisterEg}. There, the 0-framed unknots originally bound compressing discs in the handlebody, which are then doubled. 
Observe that after isotoping, the resulting link diagram has exactly the form claimed. Figure~\ref{Fig:BorrSisterEg} shows an example for one of the Borromean twisted sisters; the other links in the Borromean family can be analogously verified. 

\begin{figure}
\centering
\includegraphics{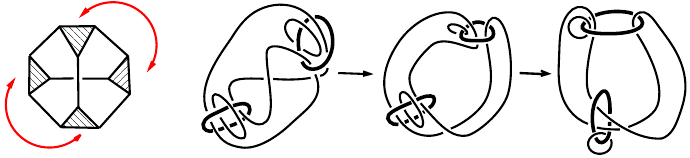}
\caption{A fundamental shadow link determined from one of the Borromean sisters. Black arrows indicate isotopy. The 0-framed unknots are indicated by a thicker line to avoid overcrowding the diagram.}
\label{Fig:BorrSisterEg}
\end{figure}

Next suppose that $G$ is obtained by $n$ graph moves on one of the Borromean basic graphs. Let $L'$ be the fully augmented link after $n-1$ graph moves with gluing graph $G'$ and corresponding fundamental shadow link $\Tilde{L'}$. Consider the local diagram obtained from the gluing of tetrahedra according to the graph move; see \reffig{LocalPicFromGluing} and \reffig{DiagramGraphMove}.

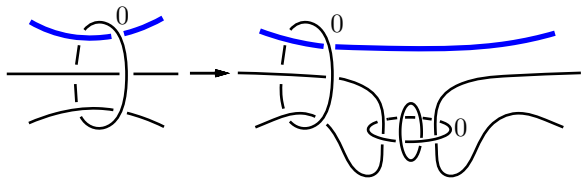
\begin{figure}
  \centering
\begingroup%
  \makeatletter%
  \providecommand\color[2][]{%
    \errmessage{(Inkscape) Color is used for the text in Inkscape, but the package 'color.sty' is not loaded}%
    \renewcommand\color[2][]{}%
  }%
  \providecommand\transparent[1]{%
    \errmessage{(Inkscape) Transparency is used (non-zero) for the text in Inkscape, but the package 'transparent.sty' is not loaded}%
    \renewcommand\transparent[1]{}%
  }%
  \providecommand\rotatebox[2]{#2}%
  \newcommand*\fsize{\dimexpr\f@size pt\relax}%
  \newcommand*\lineheight[1]{\fontsize{\fsize}{#1\fsize}\selectfont}%
  \ifx\svgwidth\undefined%
    \setlength{\unitlength}{220.32045937bp}%
    \ifx\svgscale\undefined%
      \relax%
    \else%
      \setlength{\unitlength}{\unitlength * \real{\svgscale}}%
    \fi%
  \else%
    \setlength{\unitlength}{\svgwidth}%
  \fi%
  \global\let\svgwidth\undefined%
  \global\let\svgscale\undefined%
  \makeatother%
  \begin{picture}(1,0.320779)%
    \lineheight{1}%
    \setlength\tabcolsep{0pt}%
    \put(0,0){\includegraphics[width=\unitlength,page=1]{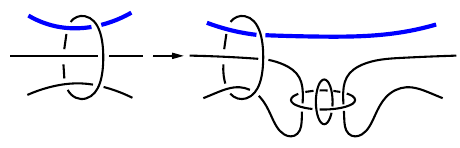}}%
    \put(0.20700934,0.28346922){\color[rgb]{0,0,0}\makebox(0,0)[lt]{\lineheight{1.25}\smash{\begin{tabular}[t]{l}$0$\end{tabular}}}}%
    \put(0.77959612,0.09234919){\color[rgb]{0,0,0}\makebox(0,0)[lt]{\lineheight{1.25}\smash{\begin{tabular}[t]{l}$0$\end{tabular}}}}%
    \put(0.57017958,0.26926991){\color[rgb]{0,0,0}\makebox(0,0)[lt]{\lineheight{1.25}\smash{\begin{tabular}[t]{l}$0$\end{tabular}}}}%
  \end{picture}%
\endgroup%

  \caption{How graph move affects diagram.}
  \label{Fig:DiagramGraphMove}
\end{figure}

By our construction, the edge $e$ in $G'$ which was deleted in the $n^{th}$ graph move was a crossing circle edge. Hence, by induction, one of the strands in the left diagram of \reffig{DiagramGraphMove} forms part of an unknotted component, without loss of generality the thicker blue strand at the top. Since each truncated tetrahedron is added so that one of the shaded faces glues to match the gluing in $G'$, the new tetrahedron will be glued so that an edge joining two shaded faces is in the position of the blue edge, preserving the unknotted component. See also \reffig{LocalPicFromGluing}. Gluing the shaded faces on the new tetrahedron according to the loop piece as in \reffig{LocalPicFromGluing} gives the right diagram in \reffig{DiagramGraphMove}. Again the 0-framed unknots indicate doubled compression discs. Observe that we obtain the same local diagram as the local diagram obtained from one central subdivision on the nerve; compare \reffig{FALlocalPic}. (In these examples, the gluing on all edges is the identity.) 
By induction $\#^{n+1}(S^1\times S^2)\setminus\Tilde{L'}$ has the presentation claimed. 
\end{proof}


\section{Coloured Jones polynomials for augmented links} \label{Sec:CJPforOFAL}

The purpose of this section is to give a new calculation of coloured Jones polynomials for octahedral fully augmented links. Our methods use recoupling theory. We note that Van der Veen has calculations that are similar in~\cite{VanDerVeen:FALvolConj}, but his focus is a different colouring, a different root of unity, and a different volume conjecture. Note also that Van der Veen allows many more augmentation rings than we do in the links that he considers.

\subsection{Recoupling Theory and Coloured Jones Polynomials} \label{Sec:recoup}

We first review background and definitions required to compute coloured Jones polynomials, using Jones--Wenzl idempotents and Kauffman brackets.

\begin{definition}[Coloured Jones polynomial]\label{Def:CJP}
Let $L$ be an oriented link in $S^3$ with $n$ ordered components. Let $D$ be a diagram of $L$ equipped with blackboard framing. Denoted by $\text{lk}_{ij}(D)$ the linking number of $D_i$ and $D_j$, and by $\text{lk}_{ii}(D)$ the writhe of $D_i$. Let $w_i(D)= \sum_{j=1}^n \text{lk}_{ij}(D)$.
The $(\pmb{i}+\pmb{1})^{th}$ coloured Jones polynomial of $L$ is defined as
\[
J_{L,\pmb{i}+\pmb{1}}(A) =
\langle ((-1)^{i_1}A^{i^2_1+i_1})^{-w_{1}(D)}S_{i_1}(z), \dots, ((-1)^{i_n}A^{i^2_n+i_n})^{-w_n(D)}S_{i_n}(z)
\rangle_D
\]
Here $\langle \cdot, \dots, \cdot\rangle_D$ is the Kaufman multi-bracket, $S_k$ denotes the $k$-th Chebyshev polynomial of the second kind and $\pmb{i}=(i_1,i_2\dots,i_n)$ is a multi-integer with $\pmb{i}+\pmb{1}=(i_1+1,i_2+1,\dots,i_n+1)$. 
In particular, for $A\in\CC$ with $|A|=1$, $|J_{L,\pmb{i}+\pmb{1}}(A)|$ is well defined for an unoriented link $L$ and
\begin{equation}\label{Eqn:colJones}
|J_{L,\pmb{i}+\pmb{1}}(A)| =
|\langle S_{i_1}(z), \dots, S_{i_n}(z) \rangle_D|
\end{equation}
\end{definition}

We will not need the full formal definitions of the terms in \refdef{CJP}. Instead, in an attempt to keep the exposition clean and as self-contained as possible, we will present here only the tools we use, and refer to the works of Lickorish~\cite{LickorishTL},~\cite[Chapter~13 and 14]{LickorishTextbook} and Kauffman and Lins~\cite[Chapter~9]{KauffLins} for further details and further applications. A succinct description of the techniques involved can be found in \cite{MasbaumVogel}.

To compute the coloured Jones polynomial of an octahedral fully augmented link, we use results in~\cite{LickorishTL, LickorishTextbook, KauffLins}, which state that \refeqn{colJones} can be computed by decorating each component of the link with a Jones--Wenzl idempotent, and evaluating the resulting diagram in the Kauffman multi-bracket.
We represent the Jones--Wenzl idempotent by $n$ strands entering a box and $n$ strands exiting it:
\[ \vcenter{\hbox{
\begingroup%
  \makeatletter%
  \providecommand\color[2][]{%
    \errmessage{(Inkscape) Color is used for the text in Inkscape, but the package 'color.sty' is not loaded}%
    \renewcommand\color[2][]{}%
  }%
  \providecommand\transparent[1]{%
    \errmessage{(Inkscape) Transparency is used (non-zero) for the text in Inkscape, but the package 'transparent.sty' is not loaded}%
    \renewcommand\transparent[1]{}%
  }%
  \providecommand\rotatebox[2]{#2}%
  \newcommand*\fsize{\dimexpr\f@size pt\relax}%
  \newcommand*\lineheight[1]{\fontsize{\fsize}{#1\fsize}\selectfont}%
  \ifx\svgwidth\undefined%
    \setlength{\unitlength}{36bp}%
    \ifx\svgscale\undefined%
      \relax%
    \else%
      \setlength{\unitlength}{\unitlength * \real{\svgscale}}%
    \fi%
  \else%
    \setlength{\unitlength}{\svgwidth}%
  \fi%
  \global\let\svgwidth\undefined%
  \global\let\svgscale\undefined%
  \makeatother%
  \begin{picture}(1,0.34)%
    \lineheight{1}%
    \setlength\tabcolsep{0pt}%
    \put(0,0){\includegraphics[width=\unitlength,page=1]{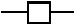}}%
    \put(0.12294433,0.21833012){\color[rgb]{0,0,0}\makebox(0,0)[lt]{\lineheight{1.25}\smash{\begin{tabular}[t]{l}$n$\end{tabular}}}}%
  \end{picture}%
\endgroup%
}} \]

As described in~\cite{MasbaumVogel}, diagrams decorated by Jones--Wenzl idempotents can be evaluated using recoupling theory~\cite{KauffLins, LickorishTextbook}.
The central object in recoupling theory is the following 3-valent vertex.

\begin{figure}[ht]
\centering
(1) 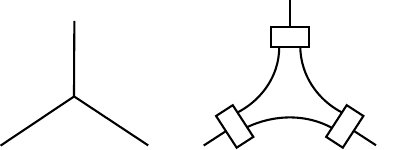
\hspace{.2in}
(2) 
\begingroup%
  \makeatletter%
  \providecommand\color[2][]{%
    \errmessage{(Inkscape) Color is used for the text in Inkscape, but the package 'color.sty' is not loaded}%
    \renewcommand\color[2][]{}%
  }%
  \providecommand\transparent[1]{%
    \errmessage{(Inkscape) Transparency is used (non-zero) for the text in Inkscape, but the package 'transparent.sty' is not loaded}%
    \renewcommand\transparent[1]{}%
  }%
  \providecommand\rotatebox[2]{#2}%
  \newcommand*\fsize{\dimexpr\f@size pt\relax}%
  \newcommand*\lineheight[1]{\fontsize{\fsize}{#1\fsize}\selectfont}%
  \ifx\svgwidth\undefined%
    \setlength{\unitlength}{72bp}%
    \ifx\svgscale\undefined%
      \relax%
    \else%
      \setlength{\unitlength}{\unitlength * \real{\svgscale}}%
    \fi%
  \else%
    \setlength{\unitlength}{\svgwidth}%
  \fi%
  \global\let\svgwidth\undefined%
  \global\let\svgscale\undefined%
  \makeatother%
  \begin{picture}(1,0.89999998)%
    \lineheight{1}%
    \setlength\tabcolsep{0pt}%
    \put(0,0){\includegraphics[width=\unitlength,page=1]{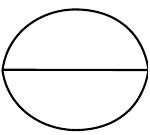}}%
    \put(0.47298665,0.73780423){\color[rgb]{0,0,0}\makebox(0,0)[lt]{\lineheight{1.25}\smash{\begin{tabular}[t]{l}$a$\end{tabular}}}}%
    \put(0.47298665,0.45895415){\color[rgb]{0,0,0}\makebox(0,0)[lt]{\lineheight{1.25}\smash{\begin{tabular}[t]{l}$b$\end{tabular}}}}%
    \put(0.47592475,0.06818959){\color[rgb]{0,0,0}\makebox(0,0)[lt]{\lineheight{1.25}\smash{\begin{tabular}[t]{l}$c$\end{tabular}}}}%
  \end{picture}%
\endgroup%

\hspace{.2in}
(3) 
\begingroup%
  \makeatletter%
  \providecommand\color[2][]{%
    \errmessage{(Inkscape) Color is used for the text in Inkscape, but the package 'color.sty' is not loaded}%
    \renewcommand\color[2][]{}%
  }%
  \providecommand\transparent[1]{%
    \errmessage{(Inkscape) Transparency is used (non-zero) for the text in Inkscape, but the package 'transparent.sty' is not loaded}%
    \renewcommand\transparent[1]{}%
  }%
  \providecommand\rotatebox[2]{#2}%
  \newcommand*\fsize{\dimexpr\f@size pt\relax}%
  \newcommand*\lineheight[1]{\fontsize{\fsize}{#1\fsize}\selectfont}%
  \ifx\svgwidth\undefined%
    \setlength{\unitlength}{72bp}%
    \ifx\svgscale\undefined%
      \relax%
    \else%
      \setlength{\unitlength}{\unitlength * \real{\svgscale}}%
    \fi%
  \else%
    \setlength{\unitlength}{\svgwidth}%
  \fi%
  \global\let\svgwidth\undefined%
  \global\let\svgscale\undefined%
  \makeatother%
  \begin{picture}(1,1)%
    \lineheight{1}%
    \setlength\tabcolsep{0pt}%
    \put(0,0){\includegraphics[width=\unitlength,page=1]{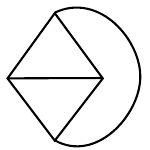}}%
    \put(0.07085698,0.65654323){\color[rgb]{0,0,0}\makebox(0,0)[lt]{\lineheight{1.25}\smash{\begin{tabular}[t]{l}$j$\end{tabular}}}}%
    \put(0.28431112,0.5014697){\color[rgb]{0,0,0}\makebox(0,0)[lt]{\lineheight{1.25}\smash{\begin{tabular}[t]{l}$k$\end{tabular}}}}%
    \put(0.58533624,0.64742129){\color[rgb]{0,0,0}\makebox(0,0)[lt]{\lineheight{1.25}\smash{\begin{tabular}[t]{l}$l$\end{tabular}}}}%
    \put(0.8152099,0.80431925){\color[rgb]{0,0,0}\makebox(0,0)[lt]{\lineheight{1.25}\smash{\begin{tabular}[t]{l}$n$\end{tabular}}}}%
    \put(0.12011563,0.18037632){\color[rgb]{0,0,0}\makebox(0,0)[lt]{\lineheight{1.25}\smash{\begin{tabular}[t]{l}$i$\end{tabular}}}}%
    \put(0.53607759,0.1822007){\color[rgb]{0,0,0}\makebox(0,0)[lt]{\lineheight{1.25}\smash{\begin{tabular}[t]{l}$m$\end{tabular}}}}%
  \end{picture}%
\endgroup%

\caption{From left to right: (1) a $3$-vertex, (2) decorated theta net $\theta(a,b,c)$, and (3) tetrahedral network $\Tet\begin{bmatrix}
        i & j & n \\
        l & m & k
    \end{bmatrix}$.}
    \label{Fig:theta-net}
\end{figure}

\begin{definition} \label{Def:3vertDef}
Assume $(a,b,c)$ is such that $a+b+c$ is even and $a\leq b+c, b\leq c+a, c\leq a+b$. Such a triple $(a,b,c)$ is called \textit{admissible}. Let $i=(a+b-c)/2$, $j=(a+c-b)/2$, $k=(b+c-a)/2$. Define an \emph{admissible $3$-vertex} to be the  Jones--Wenzl idempotents decorating the graphs given in \reffig{theta-net}~(1).
\end{definition}

Two elementary 3-valent graphs are the decorated theta net and tetrahedral network, shown in Figure~\ref{Fig:theta-net}~(2) and~(3). 

\begin{definition}\label{Def:Theta}
Define the \emph{trihedron coefficient} $\theta(a,b,c)$ to be the evaluation of the decorated theta net of \reffig{theta-net}~(2) in the Kauffman bracket. We call $(a,b,c)$ the \emph{trihedron entries} of $\theta(a,b,c)$.

Define the \emph{tetrahedron coefficient} $\Tet\begin{bmatrix}
i & j & n \\
l & m & k
\end{bmatrix}$ to be the evaluation of the decorated tetrahedral network of \reffig{theta-net}~(3) in the Kauffman bracket. 
\end{definition}

Formulas for trihedron and tetrahedron coefficients can be found in \cite[Chapter~9]{KauffLins} and \cite{MasbaumVogel}. We will not yet need explicit formulas. Instead, we will relate them to quantum $6j$-symbols.

For the rest of this paper, let $r\ge3$ be an odd integer and $q=A^2=\qroot$.
Let $I_r=\{0,1,...,r-2\}$ be the set of non-negative integers less than or equal to $r-2$.

\begin{definition}\label{Def:QuantumInt}
The \emph{quantum integer} $[n]\in\mathbb{R}$ is defined by $[n]:= (q^n-q^{-n})/(q-q^{-1})$.
The associated \textit{quantum factorial} is $[n]!:=[n][n-1]...[1]$. We define $[0]!=1$ by convention.
\end{definition}

The \emph{quantum $6j$-symbol} was introduced by Kirillov and Reshetikhin \cite{KirillovReshetikhin} from representations of the quantum group $\mathcal{U}_q(sl_2)$. It is a complex number defined for an admissible 6-tuple $(i,j,k,l,m,n)$, denoted by the symbol
$\begin{vmatrix}    i & j & k \\    l & m & n    \end{vmatrix}$. Its precise definition is postponed to \refsec{TV}. For now, we need the following.

\begin{lemma}\label{Lem:Tet}
  The quantum $6j$-symbols relate to the trihedron coefficients $\theta$ and tetrahedron coefficients $\Tet$ by the following formula.
\begin{equation}\label{Eqn:Tetrato6j}
\begin{vmatrix}
  i & j & k \\
  l & m & n
\end{vmatrix}
=\left(\sqrt{\theta(i,j,k)\theta(i,m,n)\theta(j,l,n)\theta(k,l,m)} \right)^{-1}  \Tet
\begin{bmatrix}
  i & j & n \\
  l & m & k
\end{bmatrix}
\end{equation}
By convention, when the real number $x$ is negative, we take $\sqrt{x}=\sqrt{|x|}\sqrt{-1}$.
\end{lemma}

\begin{proof}
Formulas for the evaluation of the trihedral and tetrahedral networks can be found in \cite{KauffLins, MasbaumVogel}, and compared with the definition of the quantum $6j$-symbol.
\end{proof}

\subsection{Coloured Jones and octahedral links} \label{Sec:colJones}

We next provide a new formula for the coloured Jones polynomial of any octahedral fully augmented link. The following lemmas are key to our computations.

\begin{lemma} [Merging strands, Figure~14.15 of \cite{LickorishTextbook}] \label{Lem:squishingLemma}
\[\vcenter{\hbox{
\begingroup%
  \makeatletter%
  \providecommand\color[2][]{%
    \errmessage{(Inkscape) Color is used for the text in Inkscape, but the package 'color.sty' is not loaded}%
    \renewcommand\color[2][]{}%
  }%
  \providecommand\transparent[1]{%
    \errmessage{(Inkscape) Transparency is used (non-zero) for the text in Inkscape, but the package 'transparent.sty' is not loaded}%
    \renewcommand\transparent[1]{}%
  }%
  \providecommand\rotatebox[2]{#2}%
  \newcommand*\fsize{\dimexpr\f@size pt\relax}%
  \newcommand*\lineheight[1]{\fontsize{\fsize}{#1\fsize}\selectfont}%
  \ifx\svgwidth\undefined%
    \setlength{\unitlength}{72bp}%
    \ifx\svgscale\undefined%
      \relax%
    \else%
      \setlength{\unitlength}{\unitlength * \real{\svgscale}}%
    \fi%
  \else%
    \setlength{\unitlength}{\svgwidth}%
  \fi%
  \global\let\svgwidth\undefined%
  \global\let\svgscale\undefined%
  \makeatother%
  \begin{picture}(1,0.5)%
    \lineheight{1}%
    \setlength\tabcolsep{0pt}%
    \put(0,0){\includegraphics[width=\unitlength,page=1]{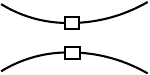}}%
    \put(0.24087552,0.39149753){\color[rgb]{0,0,0}\makebox(0,0)[lt]{\lineheight{1.25}\smash{\begin{tabular}[t]{l}$a$\end{tabular}}}}%
    \put(0.2560659,0.02692931){\color[rgb]{0,0,0}\makebox(0,0)[lt]{\lineheight{1.25}\smash{\begin{tabular}[t]{l}$b$\end{tabular}}}}%
  \end{picture}%
\endgroup%

  }}=\sum_{c}\frac{\Delta_c}{\theta(a,b,c)}
\vcenter{\hbox{
\begingroup%
  \makeatletter%
  \providecommand\color[2][]{%
    \errmessage{(Inkscape) Color is used for the text in Inkscape, but the package 'color.sty' is not loaded}%
    \renewcommand\color[2][]{}%
  }%
  \providecommand\transparent[1]{%
    \errmessage{(Inkscape) Transparency is used (non-zero) for the text in Inkscape, but the package 'transparent.sty' is not loaded}%
    \renewcommand\transparent[1]{}%
  }%
  \providecommand\rotatebox[2]{#2}%
  \newcommand*\fsize{\dimexpr\f@size pt\relax}%
  \newcommand*\lineheight[1]{\fontsize{\fsize}{#1\fsize}\selectfont}%
  \ifx\svgwidth\undefined%
    \setlength{\unitlength}{64.79999828bp}%
    \ifx\svgscale\undefined%
      \relax%
    \else%
      \setlength{\unitlength}{\unitlength * \real{\svgscale}}%
    \fi%
  \else%
    \setlength{\unitlength}{\svgwidth}%
  \fi%
  \global\let\svgwidth\undefined%
  \global\let\svgscale\undefined%
  \makeatother%
  \begin{picture}(1,0.55555557)%
    \lineheight{1}%
    \setlength\tabcolsep{0pt}%
    \put(0,0){\includegraphics[width=\unitlength,page=1]{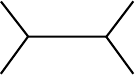}}%
    \put(0.1112542,0.44165251){\color[rgb]{0,0,0}\makebox(0,0)[lt]{\lineheight{1.25}\smash{\begin{tabular}[t]{l}$a$\end{tabular}}}}%
    \put(0.80492655,0.43717767){\color[rgb]{0,0,0}\makebox(0,0)[lt]{\lineheight{1.25}\smash{\begin{tabular}[t]{l}$a$\end{tabular}}}}%
    \put(0.1112542,0.03636945){\color[rgb]{0,0,0}\makebox(0,0)[lt]{\lineheight{1.25}\smash{\begin{tabular}[t]{l}$b$\end{tabular}}}}%
    \put(0.80854033,0.03984139){\color[rgb]{0,0,0}\makebox(0,0)[lt]{\lineheight{1.25}\smash{\begin{tabular}[t]{l}$b$\end{tabular}}}}%
    \put(0.4609101,0.3198028){\color[rgb]{0,0,0}\makebox(0,0)[lt]{\lineheight{1.25}\smash{\begin{tabular}[t]{l}$c$\end{tabular}}}}%
  \end{picture}%
\endgroup%

}}\]
where the sum is over all $c$ such that the triple $(a,b,c)$ is admissible, and the term $\Delta_n$ denotes the evaluation in the Kauffman bracket of the closure of the Jones--Wenzl idempotent in the plane. 
\end{lemma}

More information on $\Delta_n$ will be required in the next section, but we postpone it to \reflem{jbounded}.

\begin{lemma} [Removing crossing circle, Lemma~14.2 of \cite{LickorishTextbook}]\label{Lem:remove-circle}
\[
\vcenter{\hbox{
\begingroup%
  \makeatletter%
  \providecommand\color[2][]{%
    \errmessage{(Inkscape) Color is used for the text in Inkscape, but the package 'color.sty' is not loaded}%
    \renewcommand\color[2][]{}%
  }%
  \providecommand\transparent[1]{%
    \errmessage{(Inkscape) Transparency is used (non-zero) for the text in Inkscape, but the package 'transparent.sty' is not loaded}%
    \renewcommand\transparent[1]{}%
  }%
  \providecommand\rotatebox[2]{#2}%
  \newcommand*\fsize{\dimexpr\f@size pt\relax}%
  \newcommand*\lineheight[1]{\fontsize{\fsize}{#1\fsize}\selectfont}%
  \ifx\svgwidth\undefined%
    \setlength{\unitlength}{86.40000343bp}%
    \ifx\svgscale\undefined%
      \relax%
    \else%
      \setlength{\unitlength}{\unitlength * \real{\svgscale}}%
    \fi%
  \else%
    \setlength{\unitlength}{\svgwidth}%
  \fi%
  \global\let\svgwidth\undefined%
  \global\let\svgscale\undefined%
  \makeatother%
  \begin{picture}(1,0.42499998)%
    \lineheight{1}%
    \setlength\tabcolsep{0pt}%
    \put(0,0){\includegraphics[width=\unitlength,page=1]{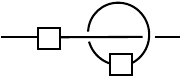}}%
    \put(0.03258591,0.23662041){\color[rgb]{0,0,0}\makebox(0,0)[lt]{\lineheight{1.25}\smash{\begin{tabular}[t]{l}$n$\end{tabular}}}}%
    \put(0.78621268,0.04621646){\color[rgb]{0,0,0}\makebox(0,0)[lt]{\lineheight{1.25}\smash{\begin{tabular}[t]{l}$a$\end{tabular}}}}%
  \end{picture}%
\endgroup%

}}=(-1)^a\left(\frac{A^{2(n+1)(a+1)}-A^{-2(n+1)(a+1)}}{A^{2(n+1)}-A^{-2(n+1)}}\right)\, \vcenter{\hbox{}}
=:\lambda_{n,a}\,\vcenter{\hbox{}}
\]
In particular, when $A$ is a $2r$-th root of unity,
\begin{equation}\label{Eqn:LambdaFormula}
  \lambda_{n,a} = \frac{(-1)^a \sin (2 \pi (n+1)(a+1)/r )}{\sin (2 \pi (n+1)/r)}.
\end{equation}
\end{lemma}

\begin{lemma}[Removing a half-twist, Theorem 3 of \cite{MasbaumVogel}] \label{Lem:removeHalfTwist}
For $(a,b,c)$ admissible, let $i=(b+c-a)/2, j=(a+c-b)/2, k=(a+b-c)/2$. 
Let $\gamma_c^{ab}=(-1)^{k}A^{k(i+j+k+2)+ij}$. Then:
    \[
\vcenter{\hbox{
\begingroup%
  \makeatletter%
  \providecommand\color[2][]{%
    \errmessage{(Inkscape) Color is used for the text in Inkscape, but the package 'color.sty' is not loaded}%
    \renewcommand\color[2][]{}%
  }%
  \providecommand\transparent[1]{%
    \errmessage{(Inkscape) Transparency is used (non-zero) for the text in Inkscape, but the package 'transparent.sty' is not loaded}%
    \renewcommand\transparent[1]{}%
  }%
  \providecommand\rotatebox[2]{#2}%
  \newcommand*\fsize{\dimexpr\f@size pt\relax}%
  \newcommand*\lineheight[1]{\fontsize{\fsize}{#1\fsize}\selectfont}%
  \ifx\svgwidth\undefined%
    \setlength{\unitlength}{43.20000172bp}%
    \ifx\svgscale\undefined%
      \relax%
    \else%
      \setlength{\unitlength}{\unitlength * \real{\svgscale}}%
    \fi%
  \else%
    \setlength{\unitlength}{\svgwidth}%
  \fi%
  \global\let\svgwidth\undefined%
  \global\let\svgscale\undefined%
  \makeatother%
  \begin{picture}(1,1.6666666)%
    \lineheight{1}%
    \setlength\tabcolsep{0pt}%
    \put(0,0){\includegraphics[width=\unitlength,page=1]{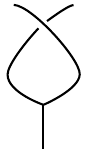}}%
    \put(0.15928202,0.73702388){\color[rgb]{0,0,0}\makebox(0,0)[lt]{\lineheight{1.25}\smash{\begin{tabular}[t]{l}$b$\end{tabular}}}}%
    \put(0.70176289,0.7335684){\color[rgb]{0,0,0}\makebox(0,0)[lt]{\lineheight{1.25}\smash{\begin{tabular}[t]{l}$a$\end{tabular}}}}%
    \put(0.5186324,0.12889231){\color[rgb]{0,0,0}\makebox(0,0)[lt]{\lineheight{1.25}\smash{\begin{tabular}[t]{l}$c$\end{tabular}}}}%
  \end{picture}%
\endgroup%

}}=\gamma_c^{ab}
\vcenter{\hbox{
\begingroup%
  \makeatletter%
  \providecommand\color[2][]{%
    \errmessage{(Inkscape) Color is used for the text in Inkscape, but the package 'color.sty' is not loaded}%
    \renewcommand\color[2][]{}%
  }%
  \providecommand\transparent[1]{%
    \errmessage{(Inkscape) Transparency is used (non-zero) for the text in Inkscape, but the package 'transparent.sty' is not loaded}%
    \renewcommand\transparent[1]{}%
  }%
  \providecommand\rotatebox[2]{#2}%
  \newcommand*\fsize{\dimexpr\f@size pt\relax}%
  \newcommand*\lineheight[1]{\fontsize{\fsize}{#1\fsize}\selectfont}%
  \ifx\svgwidth\undefined%
    \setlength{\unitlength}{43.20000172bp}%
    \ifx\svgscale\undefined%
      \relax%
    \else%
      \setlength{\unitlength}{\unitlength * \real{\svgscale}}%
    \fi%
  \else%
    \setlength{\unitlength}{\svgwidth}%
  \fi%
  \global\let\svgwidth\undefined%
  \global\let\svgscale\undefined%
  \makeatother%
  \begin{picture}(1,1.6666666)%
    \lineheight{1}%
    \setlength\tabcolsep{0pt}%
    \put(0,0){\includegraphics[width=\unitlength,page=1]{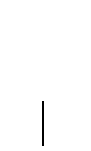}}%
    \put(0.66029934,0.8581285){\color[rgb]{0,0,0}\makebox(0,0)[lt]{\lineheight{1.25}\smash{\begin{tabular}[t]{l}$b$\end{tabular}}}}%
    \put(0.1869244,0.85812815){\color[rgb]{0,0,0}\makebox(0,0)[lt]{\lineheight{1.25}\smash{\begin{tabular}[t]{l}$a$\end{tabular}}}}%
    \put(0.5186324,0.16361453){\color[rgb]{0,0,0}\makebox(0,0)[lt]{\lineheight{1.25}\smash{\begin{tabular}[t]{l}$c$\end{tabular}}}}%
    \put(0,0){\includegraphics[width=\unitlength,page=2]{RemoveTwist2.pdf}}%
  \end{picture}%
\endgroup%

}}
\]
Moreover, reversing the sign of the crossing replaces $\gamma_c^{ab}$ by its conjugate $\bar{\gamma}_c^{ab}$, where conjugation is defined by $\bar{A}=A^{-1}$.
\end{lemma}

\begin{lemma}[Triangle pop, page~122 of \cite{KauffLins}]\label{Lem:triangleLemma}
\begin{align}
  \vcenter{\hbox{
\begingroup%
  \makeatletter%
  \providecommand\color[2][]{%
    \errmessage{(Inkscape) Color is used for the text in Inkscape, but the package 'color.sty' is not loaded}%
    \renewcommand\color[2][]{}%
  }%
  \providecommand\transparent[1]{%
    \errmessage{(Inkscape) Transparency is used (non-zero) for the text in Inkscape, but the package 'transparent.sty' is not loaded}%
    \renewcommand\transparent[1]{}%
  }%
  \providecommand\rotatebox[2]{#2}%
  \newcommand*\fsize{\dimexpr\f@size pt\relax}%
  \newcommand*\lineheight[1]{\fontsize{\fsize}{#1\fsize}\selectfont}%
  \ifx\svgwidth\undefined%
    \setlength{\unitlength}{43.20000172bp}%
    \ifx\svgscale\undefined%
      \relax%
    \else%
      \setlength{\unitlength}{\unitlength * \real{\svgscale}}%
    \fi%
  \else%
    \setlength{\unitlength}{\svgwidth}%
  \fi%
  \global\let\svgwidth\undefined%
  \global\let\svgscale\undefined%
  \makeatother%
  \begin{picture}(1,0.8333333)%
    \lineheight{1}%
    \setlength\tabcolsep{0pt}%
    \put(0,0){\includegraphics[width=\unitlength,page=1]{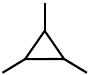}}%
    \put(0.53487422,0.66250938){\color[rgb]{0,0,0}\makebox(0,0)[lt]{\lineheight{1.25}\smash{\begin{tabular}[t]{l}$l$\end{tabular}}}}%
    \put(0.01960276,0.15675714){\color[rgb]{0,0,0}\makebox(0,0)[lt]{\lineheight{1.25}\smash{\begin{tabular}[t]{l}$p$\end{tabular}}}}%
    \put(0.85971921,0.16235753){\color[rgb]{0,0,0}\makebox(0,0)[lt]{\lineheight{1.25}\smash{\begin{tabular}[t]{l}$q$\end{tabular}}}}%
    \put(0.27051997,0.34606574){\color[rgb]{0,0,0}\makebox(0,0)[lt]{\lineheight{1.25}\smash{\begin{tabular}[t]{l}$i$\end{tabular}}}}%
    \put(0.43966098,-0.02133665){\color[rgb]{0,0,0}\makebox(0,0)[lt]{\lineheight{1.25}\smash{\begin{tabular}[t]{l}$j$\end{tabular}}}}%
    \put(0.63008746,0.34886619){\color[rgb]{0,0,0}\makebox(0,0)[lt]{\lineheight{1.25}\smash{\begin{tabular}[t]{l}$k$\end{tabular}}}}%
  \end{picture}%
\endgroup%

  }}
=& \sqrt{\frac{\theta(p,i,j)\theta(k,q,j)\theta(i,k,l)}{\theta(p,q,l)}}\begin{vmatrix}
    p & i & l \\
    k & q & j
\end{vmatrix}\vcenter{\hbox{
\begingroup%
  \makeatletter%
  \providecommand\color[2][]{%
    \errmessage{(Inkscape) Color is used for the text in Inkscape, but the package 'color.sty' is not loaded}%
    \renewcommand\color[2][]{}%
  }%
  \providecommand\transparent[1]{%
    \errmessage{(Inkscape) Transparency is used (non-zero) for the text in Inkscape, but the package 'transparent.sty' is not loaded}%
    \renewcommand\transparent[1]{}%
  }%
  \providecommand\rotatebox[2]{#2}%
  \newcommand*\fsize{\dimexpr\f@size pt\relax}%
  \newcommand*\lineheight[1]{\fontsize{\fsize}{#1\fsize}\selectfont}%
  \ifx\svgwidth\undefined%
    \setlength{\unitlength}{43.20000172bp}%
    \ifx\svgscale\undefined%
      \relax%
    \else%
      \setlength{\unitlength}{\unitlength * \real{\svgscale}}%
    \fi%
  \else%
    \setlength{\unitlength}{\svgwidth}%
  \fi%
  \global\let\svgwidth\undefined%
  \global\let\svgscale\undefined%
  \makeatother%
  \begin{picture}(1,0.8333333)%
    \lineheight{1}%
    \setlength\tabcolsep{0pt}%
    \put(0,0){\includegraphics[width=\unitlength,page=1]{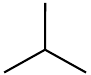}}%
    \put(0.55651591,0.62742278){\color[rgb]{0,0,0}\makebox(0,0)[lt]{\lineheight{1.25}\smash{\begin{tabular}[t]{l}$l$\end{tabular}}}}%
    \put(0.01947814,0.13490533){\color[rgb]{0,0,0}\makebox(0,0)[lt]{\lineheight{1.25}\smash{\begin{tabular}[t]{l}$p$\end{tabular}}}}%
    \put(0.82364333,0.14325411){\color[rgb]{0,0,0}\makebox(0,0)[lt]{\lineheight{1.25}\smash{\begin{tabular}[t]{l}$q$\end{tabular}}}}%
  \end{picture}%
\endgroup%
  
}}\label{Eqn:Triangleto6j}
    \end{align} 
\end{lemma}

\begin{definition}
    Given a 3-valent vertex, define the \textit{triangle move} to be the move shown in \reffig{CentralSub-TriangleMove} right. Define the \textit{triangle pop} to be the reverse of the triangle remove, shown in \refeqn{Triangleto6j}.
For the configuration on the left of \refeqn{Triangleto6j}, we refer to $l$ as the \textit{head} and $j$ as the \textit{base} of the triangle pop. These are also the edges coloured red on the right of \reffig{CentralSub-TriangleMove}. 
\end{definition}

\begin{figure}[ht]
    \centering
    \includegraphics{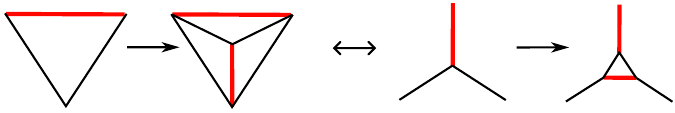}
    \caption{Triangle move is dual to central subdivision.}
    \label{Fig:CentralSub-TriangleMove}
\end{figure}

We remark that the triples involved in the trihedron coefficients in \refeqn{Triangleto6j} are precisely the triples corresponding to the four vertices of the tetrahedral net obtained from the left of \refeqn{Triangleto6j} by connecting the strands coloured $p,q,l$ into a fourth 3-vertex.

By considering the dual graph of a fully augmented link instead of its nerve, one obtains the following corollary to \refprop{octaLinksCharac}.

\begin{corollary}[Corollary to \refprop{octaLinksCharac}] \label{Cor:dualGraph}
    Let $L$ be a fully augmented link with $c$ crossing circles and complement decomposing into two polyhedra isometric to $P$. Then $P$ is obtained by gluing $c-1$ regular ideal octahedra if and only if the dual graph of $L$ is obtained by successive triangle moves on the complete graph on four vertices. In this case, there are $c-2$ triangle moves.
\end{corollary}

\begin{proof}
The dual graph of the complete graph on four vertices is still the complete graph on four vertices. Altering a triangle of the nerve by central subdivision alters the dual graph by a triangle move; see \reffig{CentralSub-TriangleMove}. Now, a single central subdivision corresponds to adding exactly one octahedron. Thus, $c-2$ subdivisions are required to obtain $c-1$ octahedra comprising $P$, noting that if the nerve of $L$ is the complete graph on four vertices, then $P$ is a single octahedron. This corresponds to $c-2$ triangle moves on the dual graph.
\end{proof}

\begin{proposition} \label{Prop:ColJonesforOctFAL}
   Let $L$ be an octahedral fully augmented link of $n$ components with $c$ crossing circles and diagram $D$. Decorate $D$ by $\pmb{i}=(a_1,a_2\dots,a_c,i_1,i_2,\dots,i_{n-c})$, where $a_1,\dots,a_c$ colour the crossing circles and $i_1,\dots,i_{n-c}$ colour the knot strands. 
Let $H\subseteq\{a_1,a_2,\dots,a_c\}$ be the colours of the crossing circles adjacent to half-twists; $H=\varnothing$ if $L$ has no half-twists.
Then the Kauffman multi-bracket $\langle S_{a_1}(z),\dots,S_{a_c}(z),S_{i_1}(z),\dots,S_{i_{n-c}}(z)\rangle_D$ is:
\begin{equation}
  \langle S_{a_1}(z),\dots,S_{a_c}(z),S_{i_1}(z),\dots,S_{i_{n-c}}(z) \rangle_D = \sum_{j_1,j_2,\dots,j_c}N_{\pmb{i},j_1,j_2,\dots,j_c}, \label{Eqn:colJonesFormula}
\end{equation}
where each $j_l$ is a summation variable arising from merging the strands passing between the crossing circle coloured $a_l$ according to \reflem{squishingLemma}, and $N_{\pmb{i},j_1,j_2,\dots,j_c}$ is a product of:
\begin{enumerate}
\item[(1)] $\Delta_{j_l}\lambda_{j_l,a_l}$ for $1\leq l\leq c$,
  \vspace{.1in}
\item[(2)] $\gamma_{j_l}^{i_k,i_m}$ for each $l\in\{1,2,\dots,c\}$ such that $a_l\in H$ and $i_k,i_m$ are the colours of the knot strands passing through the crossing circle coloured $a_l$,
\item[(3)] $c-1$ quantum $6j$-symbols, each of the form
  $\begin{vmatrix}
  i_q & i_p & j_k \\
  i_r & i_s & j_l
\end{vmatrix}$,
  for some $q,p,r,s\in\{1,\dots,n-c\}$, $k,l\in\{1,\dots,c\}$, and every summation variable appears at least once in such a quantum $6j$-symbol.
\end{enumerate}
\end{proposition}

\begin{proof}
We compute $\langle S_{a_1}(z),\dots,S_{i_{n-c}}(z)\rangle_D$ as follows. 
At each crossing circle apply \reflem{squishingLemma} to the strands passing between the crossing circle, then remove the crossing circle by \reflem{remove-circle}. If necessary, remove half-twists by \reflem{removeHalfTwist}. The first step contributes a sum of the form $\sum_j\frac{\Delta_j}{\theta(i,k,j)}$ where $j$ is the colour on the merged strand and $i$ and $k$ are the colours on the knot strand passing through that crossing circle. This gives the term $\Delta_{j_l}$ in~(1). The second step contributes a factor of $\lambda_{j,a}$, where $a$ is the colour of the crossing circle. This gives $\lambda_{j_l,a_l}$ in~(1). If two strands coloured $i,k$ have a half-twist, then by \reflem{removeHalfTwist} removing the half-twist contributes a factor of $\gamma_{j}^{i,k}$, where $j$ is a summation variable. This gives the terms in~(2). 
Observe that after performing these operations we are left with a network $G'_{\pmb{i}}$ that is exactly the network obtained by decorating the dual graph $G'$ of $L$ with the relevant Jones-Wenzl idempotents: We have exactly reversed the process shown in \reffig{dual}. 

By \refcor{dualGraph}, after performing $c-2$ triangle pops on $G'_{\pmb{i}}$, we obtain a tetrahedral network, and we complete the evaluation using \reflem{Tet}. This proves there are $c-1$ quantum $6j$-symbols.

Consider the configuration for a triangle pop in the trivalent dual graph. Recall that every vertex in the dual graph is adjacent to exactly one dimer edge. The edges in the dimer are precisely the edges which are coloured by the summation variables in $G'_{\pmb{i}}$ after we have applied \reflem{squishingLemma}, \reflem{remove-circle} and \reflem{removeHalfTwist}.
It follows that every triple around a vertex in $G'_{\pmb{i}}$ consists of exactly one summation variable. Moreover, this property of triples is preserved by triangle pops; see \reflem{triangleLemma}.
It follows that in all of the $c-2$ triangle pops which are applied to $G'_{\pmb{i}}$, there are always precisely two summation variables involved. In our decoration of the network before the triangle pop, we can always take these summation variables to be in the position of the head and base. Moreover, the final tetrahedral network will have exactly one outer edge and one opposite inner edge coloured by summation variables, as in \reffig{CentralSub-TriangleMove}, left. This proves that the form of the quantum $6j$-symbol is as in~(3). In the process of obtaining $G'$ from the complete graph on four vertices by successive triangle moves, every edge in the dimer arises either in a triangle move or it is in the complete graph on four vertices to begin with. Hence every summation variable appears at least once in a quantum $6j$-symbol.

We have now shown that all the terms in~(1), (2), and~(3) appear as claimed. It remains to show that no other terms arise. The only other terms that occur in the process are trihedron coefficients. We show that the product of all trihedron coefficients is one. 
Trihedron coefficients arise from merging strands in \reflem{squishingLemma}, from triangle pops \reflem{triangleLemma}, and from the evaluation of the final tetrahedral network; no other trihedron coefficients arise.

In \reflem{squishingLemma}, trihedron entries are a triple of two colours from knot strands and a summation variable. In \reflem{triangleLemma}, using the observation above that exactly the head and base are labelled with summation variables, the trihedron entries that arise must be a triple of two colours from knot strands and one summation variable. Finally, when evaluating the final tetrahedron network using \reflem{Tet}, again by the observation that exactly dimer edges are coloured by the summation variables, it follows that the trihedron entries that appear have exactly two entries that are colours from knot strands, and exactly one summation variable. 
Thus, it suffices to show that for each summation variable $j$, the product of all trihedron coefficients containing $j$ is equal to one.

First suppose $c=2$. Then $G'_{\pmb{i}}$ is a tetrahedral net and $L$ is the Borromean rings or one of the Borromean twisted sisters as in  \reffig{BorrRingsandSisters}. Hence, $L$ has either one knot strand coloured $i$, or in the case of two half-twists, it has two knot strands coloured $i$ and $k$. Merging in \reflem{squishingLemma} introduces terms of the form $(\theta(i,k,j_1)\theta(i,k,j_2))^{-1}$, where $j_1$, $j_2$ are the summation variables; set $i=k$ if there is just one knot strand. Then applying \reflem{Tet} contributes
\[ \sqrt{\theta(i,k,j_1)\theta(i,k,j_1)\theta(i,k,j_2)\theta(i,k,j_2)} = \theta(i,k,j_1)\theta(i,k,j_2); \]
these trihedron coefficients cancel with $(\theta(i,k,j_1)\theta(i,k,j_2))^{-1}$ from \reflem{squishingLemma}. 

Next, suppose $c>2$. After applying \reflem{squishingLemma}, \reflem{remove-circle} and \reflem{removeHalfTwist}, each summation variable $j$ appears in \reflem{squishingLemma} meeting two 3-vertices labelled $(i,k,j)$, where $i$, $k$ colour knots strands. This term is multiplied by $\theta(i,k,j)^{-1}$, from \reflem{squishingLemma}.

Suppose $j$ is involved in a triangle pop as a base edge. By \reflem{triangleLemma}, the contribution is $\sqrt{\theta(i,k,j)\theta(i,k,j)} = \theta(i,k,j)$, and so this cancels with the trihedron coefficient from \reflem{squishingLemma}. The triangle popping removes $j$ from the diagram, so it contributes to no further trihedron coefficients.

Now suppose the first triangle pop involving $j$ is with $j$ as a head.
There may be a sequence of triangle pops with $j$ as head, and each contributes two trihedron coefficients involving $j$. The first coefficient is of the form $(\theta(x,y,j))^{1/2}$, where $(x,y,j)$ decorates a 3-vertex that is removed by the triangle pop. The other coefficient has the form $(\theta(p,q,j))^{-1/2}$, where $(p,q,j)$ decorates the new 3-vertex that is created by the triangle pop.

Consider additional terms involving $(p,q,j)$. There will be exactly one other such term, arising when this 3-vertex is removed, either by a triangle pop with $j$ as head, or as base, or by evaluating the final tetrahedron network. In all three cases, since $(p,q,j)$ is an existing vertex, it contributes $(\theta(p,q,j))^{1/2}$, which cancels with the $(\theta(p,q,j))^{-1/2}$ contributed at the initial creation of the 3-vertex. This argument also shows that if a triangle pop removes a 3-vertex that was created by a previous triangle pop, then the numerator term always cancels with a denominator from a previous step. So if a triangle pop is not the first to involve the edge coloured $j$, then the only $\theta$ contribution involving $j$ will contribute to the denominator, which subsequently cancels by the argument of this paragraph.

So consider the first triangle pop. Because the edge coloured $j$ arose by a merge move from \reflem{squishingLemma}, there are two 3-vertices meeting $j$ labeled $(i,k,j)$ at this initial step. One 3-vertex is removed under the first triangle pop, with $j$ as head by assumption, contributing $(\theta(i,k,j))^{1/2}$. Since the merge move already contributed $(\theta(i,k,j))^{-1}$, these simplify to give a contribution of $(\theta(i,k,j))^{-1/2}$. Then the argument proceeds as in the previous paragraph. At some stage, the other 3-vertex with colours $(i,k,j)$ will be removed, by a triangle pop with $j$ as head, or as base, or by the evaluation of the final tetrahedron network. As above, the removal contributes $(\theta(i,k,j))^{1/2}$, which cancels. 
\end{proof}

\section{Turaev--Viro invariants}\label{Sec:TV}
This section provides a new proof of \refconj{TVconj} for flat octahedral fully augmented links.

We utilise a formula of Detcherry, Kalfagianni and Yang \cite[Theorem~1.1]{DKYRTTV}, to relate the Turaev--Viro invariants of an octahedral fully augmented link complement to its coloured Jones polynomial:

\begin{theorem} [Theorem~1.1 of \cite{DKYRTTV}]\label{Thm:TVfromColJones}
Let $L$ be a link in $S^3$ with $n$ components.
Let $r=2m+1\ge3$ be an odd integer and $A$ a primitive $2r^{th}$ root of unity, with $q=A^2$. Then 
\begin{equation} \label{Eqn:TVasSumofJones}
TV_r(S^3\setminus L,q)=2^{n-1}(\eta_r')^2\sum_{1\leq\pmb{i}\leq m}|J_{L,\pmb{i}}(A)|^2.
\end{equation}
Here $\eta_r' = (A^2-A^{-2})/\sqrt{-r}$ and $\pmb{i}=(i_1,\dots,i_n)$ and $1\leq\pmb{i}\leq m$ means $1\leq i_k\leq m$ for each $k$.
\end{theorem}

\begin{remark}
Note that \cite{DKYRTTV} uses a different normalisation of the coloured Jones polynomial than we used in the previous section, following \cite{KauffLins,LickorishTextbook}. Our definition takes into account a correction factor for the framing of the links, and differs from that of \cite{DKYRTTV} by a power of $A$. However, since $A$ is a root of unity and we are only interested in the modulus of $|J_{L,\pmb{i}}|$ in \refthm{TVfromColJones}, our formula will suffice.
\end{remark}

We will use \refthm{TVfromColJones}  with \refprop{ColJonesforOctFAL} and \refeqn{colJones} to evaluate limits of $TV_r$. By \refprop{ColJonesforOctFAL}, we will have terms involving $\Delta_j$, $\lambda_{j,a}$, $\gamma_j^{i,k}$, and quantum 6j-symbols. We will need further information on these terms to find and apply bounds.

We first consider $\Delta_j$ and $\lambda_{j,a}$. Recall $\Delta_j$ is the closure of the Jones--Wenzl idempotent in the plane evaluated in the Kauffman bracket, and $\lambda_{j,a}$ is defined in \reflem{remove-circle}. 

\begin{lemma}[Lemma~4 of \cite{KauffLins}]\label{Lem:jbounded}
The term $\Delta_{j} = (-1)^{j} [j+1]$, where $[n]$ is the quantum integer of \refdef{QuantumInt}. It is the $j^{th}$ Chebyshev polynomial of second kind in variable $(-A^2-A^{-2})$:
\[ \Delta_j=\frac{(-1)^j(A^{2(j+1)}-A^{-2(j+1)})}{A^2-A^{-2}} \]
If $q=A^2$ is a primitive $r$-th root of unity, where $r\ge3$ is odd, then $\Delta_{r-1}=0$ and for all $0\leq j\leq r-2,$ 
\begin{equation}\label{Eqn:Delta_n}
\Delta_j = (-1)^j \frac{\sin ( 2\pi (j+1)/r)}{\sin(2\pi/r)}\neq 0.
\end{equation}
\end{lemma}

A quantum $6j$ symbol is associated with an admissible 6-tuple $(i,j,k,l,m,n)$, as follows.

\begin{definition} \label{Def:r-admis}
A triple $(i,j,k)$ of elements of $I_r$ is called \emph{$r$-admissible} (or \emph{admissible}), if
\begin{enumerate}
\item $i+j+k$ is even,
\item $i\leq j+k$, $j\leq k+i$, $k\leq i+j$, and
\item $i+j+k\leq 2r-4$.
\end{enumerate}
A 6-tuple $(i,j,k,l,m,n)$ of elements of $I_r$ is \textit{admissible} if each of the triples $(i,j,k)$, $(j,l,n)$, $(i,m,n)$, $(k,l,m)$ are $r$-admissible.
\end{definition}

Throughout, let $n_r=\frac{r-1}{2}$ if $r\equiv1\mod4$ and $n_r=\frac{r-3}{2}$ if $r\equiv3\mod4$. 
Note that for $j\in I_r$, if $(n_r,n_r,j)$ is admissible, then the admissibility conditions imply that $j\in\{0,2,4,\dots,r-3\}$.

\begin{lemma} \label{Lem:DeltaLambdaSigns}
Let $n_r=(r-1)/2$ if $r\equiv 1\mod 4$ and $n_r=(r-3)/2$ if $r\equiv 3\mod 4$.
For $j,j_1,j_2\in I_r$ such that $(n_r,n_r,j),(n_r,n_r,j_1),(n_r,n_r,j_2)$ are $r$-admissible, the following hold.
\begin{enumerate}
\item The sign of $\Delta_j$ is positive if $0\leq j\leq (r-3)/2$ and negative if $(r-1)/2\leq j\leq r-2$.

\item The sign of $\lambda_{j,n_r}$, defined as in \refeqn{LambdaFormula}, is
  \[ \sgn(\lambda_{j,n_r})=\begin{cases}
  \sgn(\Delta_j), & r\equiv3\mod4, \\
  -\sgn(\Delta_j), & r\equiv1\mod4.
  \end{cases}\]
\end{enumerate} 
\end{lemma}

\begin{proof}
For (1), by \reflem{jbounded},
\[ \Delta_j = (-1)^j \, \frac{\sin(2\pi(j+1)/r)}{\sin(2\pi/r)} \in \RR \]
By admissibility conditions on $(n_r, n_r, j)$, the integer $j$ must be even. The denominator $\sin(2\pi/r)$ is positive for all $r\ge3$.
The term $(2\pi(j+1))/r$ lies in $(0,\pi)$ for $0\leq j\leq (r-3)/2$, and lies in $(\pi,2\pi)$ for $(r-1)/2\leq j\leq r-2$. The statement of (1) follows. 

For (2), by \refeqn{LambdaFormula}, we have
\[ \lambda_{j,n_r} = \frac{\sin(2\pi(j+1)(n_r+1)/r)}{\sin(2\pi(j+1)/r)}\in\RR
\]
Now,
\[\frac{2\pi}{r}(j+1)(n_r+1) = \begin{cases}
  \pi(j+1)-\frac{\pi}{r}(j+1), & r\equiv 3\mod4, \\
  \pi(j+1)+\frac{\pi}{r}(j+1), & r\equiv 1\mod4.
\end{cases}\]
By the argument of (1), $\frac{\pi}{r}(j+1) \in (0,\frac{\pi}{2})$ for $0\leq j\leq (r-3)/2$ and $\frac{\pi}{r}(j+1) \in (\frac{\pi}{2},\pi)$ for $(r-1)/2\leq j\leq r-2$.
Note also that $j+1$ is odd since $j$ is even by admissibility. Hence, if $r\equiv 3\mod4$, then 
$\frac{2\pi}{r}(j+1)(n_r+1) \in  (j \pi , \pi (j+1)-\frac{\pi}{2}) \cup (\pi (j+1)-\frac{\pi}{2}, \pi (j+1))$, which is in the first and second quadrant of the unit circle. Therefore, $\sin\left(\frac{2\pi}{r}(j+1)(n_r+1)\right)>0$, so the sign of $\lambda_{j,n_r}$ is given by $\sgn \left(\sin\left(2\pi(j+1)/r\right)\right) = \sgn(\Delta_j)$. If $r\equiv 1\mod4$, then $\frac{2\pi}{r}(j+1)(n_r+1)$ is in the third and fourth quadrant and the result follows analogously.
\end{proof}

We need to determine the signs of the quantum $6j$ symbols as well. To do so, we need more information on these symbols. We now recall their full definition, and related results.

\begin{definition} \label{Def:6j-defn}
The \emph{quantum $6j$-symbol} associated with an admissible 6-tuple $(i,j,k,l,m,n)$ is the complex number
\begin{equation}
  \begin{vmatrix}
    i & j & k \\
    l & m & n
  \end{vmatrix} =(-1)^{(i+j+k+l+m+n)/2}\Delta(ijk)\Delta(imn)\Delta(ljn)\Delta(lmk)  \sum_{z=\max\{T_1,T_2,T_3,T_4\}}^{\min\{Q_1,Q_2,Q_3\}}S_z
\end{equation}
Here $\Delta(ijk)$ is given by
\begin{equation}\label{Eqn:DeltaEqn}
  \Delta(ijk)  = \left(\frac{[\frac{i+j-k}{2}]![\frac{i-j+k}{2}]![\frac{j+k-i}{2}]!}{[\frac{i+j+k}{2}+1]!} \right)^{1/2}
\end{equation}
The term $S_z$ is defined as
\begin{equation}\label{Eqn:sz}
S_z = \frac{(-1)^z[z+1]!}{[z-T_1]![z-T_2]![z-T_3]![z-T_4]![Q_1-z]![Q_2-z]![Q_3-z]!} \end{equation}
where the values $T_p$ correspond to faces of the tetrahedron as in \reffig{6jtetra}:
\begin{equation}\label{Eqn:Tdefn}
  T_1=\frac{i+j+k}{2}, \quad T_2=\frac{j+l+n}{2}, \quad T_3=\frac{i+m+n}{2}, \quad T_4=\frac{k+l+m}{2},
\end{equation}
The values $Q_q$ correspond to quadrilaterals separating two pairs of vertices as in \reffig{6jtetra}:
\begin{equation}\label{Eqn:Qdefn}
  Q_1=\frac{i+j+l+m}{2}, \quad Q_2=\frac{i+k+l+n}{2}, \quad Q_3=\frac{j+k+m+n}{2}
\end{equation}
\end{definition}

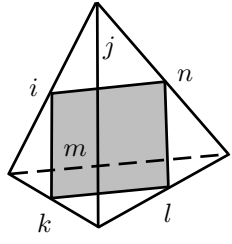
\begin{figure}
    \centering
    \[
\vcenter{\hbox{
\begin{overpic}[scale=1.2]{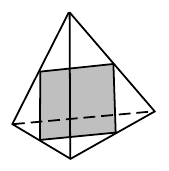}
\put(15, 57){$i$}
\put(43,73){$j$}
\put(18,5){$k$}
\put(66, 9){$l$}
\put(28,34){$m$}
\put(71,63){$n$}
\end{overpic} }}
\]
    \caption{Each of $T_1$, $T_2$, $T_3$, $T_4$ corresponds to a face of the tetrahedron. Each of $Q_1$, $Q_2$. $Q_3$ corresponds to a quadrilateral separating two pairs of vertices.}
    \label{Fig:6jtetra}
\end{figure}

In~\cite[Theorem~A.1]{MainPaper}, the authors prove the following result of Costantino~\cite{Cost6jSymb} for the root of unity $\qroot$. The form of the theorem that we need appears in~\cite{KumarMelbyTVadditivity}.

\begin{theorem}[Theorem~3.4~(1) of \cite{KumarMelbyTVadditivity}]\label{Thm:6TuplesandHypIdealTetra} 
If $(n_1^{(r)},n_2^{(r)},n_3^{(r)},n_4^{(r)},n_5^{(r)},n_6^{(r)})$ is a sequence of admissible 6-tuples with
\begin{enumerate}[label=(\roman*)]
\item $0\leq Q_j-T_i\leq (r-2)/2$ for $1\leq i\leq 4$, $1\leq j\leq3$, and
\item $(r-2)/2\leq T_i\leq r-2$ for $1\leq i\leq4$,
\end{enumerate}
then
for each $r$, the sign of $S_z$ is independent of the choice of $z$, for $\max\{T_i\}\leq z\leq\min\{Q_j\}$.
\end{theorem}

\begin{lemma}\label{Lem:6jSign}
Let $n_r=(r-1)/2$ if $r\equiv 1\mod 4$, and $n_r=(r-3)/2$ if $r\equiv 3\mod 4$. Then the quantum $6j$-symbol
\[\begin{vmatrix}
  n_r & n_r & j_1 \\
  n_r & n_r & j_2
\end{vmatrix}\]
is real-valued. Its sign is positive if $r\equiv3\mod4$ and negative if $r\equiv1\mod4$.
\end{lemma}

\begin{proof}
The given quantum $6j$-symbol is real-valued due to \cite[Lemma~4.7]{KumarMelbyTVadditivity}. The rest of the proof considers its sign, appealing to \refdef{6j-defn}. 

For the admissible 6-tuple $(n_r, n_r, j_1, n_r, n_r, j_2)$, the values of $T_i$ and $Q_j$ are as follows:
\[ T_1=T_4 = n_r+ \frac{j_1}{2}, \quad T_2=T_3=n_r+\frac{j_2}{2}, \quad Q_1=2n_r, \quad Q_2=Q_3=n_r + \frac{j_1+j_2}{2}
\]
Thus $Q_j-T_i$ is one of $n_r-(j_1/2)$, $n_r-(j_2/2)$, $j_1/2$, or $j_2/2$. Note also that by admissibility, $j_1, j_2\in \{0, 2, 4, \dots, r-3\}$. 
Then by \refdef{6j-defn}, 
\begin{equation}\label{Eqn:6jForSign}
\begin{vmatrix}
n_r & n_r & j_1 \\
n_r & n_r & j_2
\end{vmatrix}
=(-1)^{(j_1 + j_2)/2} \Delta(n_r,n_r,j_1)^2 \Delta(n_r,n_r,j_2)^2
\sum_{z=n_r+\frac12 \max\{j_1,j_2\}}^{\min\{2n_r,n_r+\frac12(j_1+j_2)\}} S_z^{j_1,j_2}
\end{equation}
where
\begin{equation} \label{Eqn:nr6j}
S_z^{j_1,j_2}=\frac{(-1)^z[z+1]!}{([z-n_r-\frac{j_1}{2}]!)^2([z-n_r-\frac{j_2}{2}]!)^2[2n_r-z]!([n_r-\frac{j_1}{2}-\frac{j_2}{2}-z]!)^2}
\end{equation}

Consider the case where $j_1,j_2\neq0$. We show that the assumptions of \refthm{6TuplesandHypIdealTetra} hold if $j_1,j_2\neq0$. 
By admissibility, for $l\in\{1,2\}$, we have $j_l/2 \geq 0$ and $j_l/2 \leq n_r$, hence $n_r-(j_l/2)\geq 0$. This gives the lower inequality for assumption~(i).  For the upper, $j_l\leq r-3$ so $(j_l/2)\leq (r-3)/2< (r-2)/2$. Also, $n_r-(j_l/2)\leq n_r=(r-3)/2$ if $r\equiv3\mod4$. If $r\equiv 1\mod 4$, since $j_l\neq 0$, we have $n_r-(j_l/2)\leq (r-1)/2 - 1/2$. This gives the upper bound of~(i).

For assumption~(ii) of \refthm{6TuplesandHypIdealTetra}, suppose first $r\equiv1\mod4$. Then $(r-2)/2<n_r$ so $(r-2)/2\leq n_r + (j_l/2)$.
Also, $n_r+(j_l/2) \leq n_r + (r-3)/2 = r-2$. So~(ii) holds when $r\equiv1\mod4$.
Now suppose $r \equiv 3\mod 4$. Then $n_r + (j_l/2) \leq n_r + (r-3)/2 = 2n_r < r-2$, and $n_r + (j_l/2) \ge (r-2)/2$ if $j_l>0$. So~(ii) holds provided $j_l\neq0$.
Then \refthm{6TuplesandHypIdealTetra} holds when $j_l\neq 0$. 
    
By \refthm{6TuplesandHypIdealTetra}, the sign of $S_z^{j_1,j_2}$ is independent of $z$. For this case, we show the sign of $S_z^{j_1,j_2}$ is also independent of $j_1,j_2$. Possible contributions to the sign in \refeqn{nr6j} come from terms $(-1)^z$, $[z+1]!$, and $[2n_r-z]!$, since the other terms are squares of real numbers (see \reflem{jbounded}) hence always positive. Since the sign is independent of $z$, to compute the sign it suffices to find the signs of terms $[z+1]!$ and $[2n_r-z]!$ for any value of $z$ in the range of the sum of \refeqn{6jForSign}. In particular, let $j_m := \max\{j_1,j_2\}$, and consider the minimum value $z = n_r + j_m/2$. Then $[z+1]!=[n_r+(j_m/2)+1]!$ and $[2n_r-z]=[n_r-(j_m/2)]!$.

By our work checking~(i) of \refthm{6TuplesandHypIdealTetra}, we have
$0 \leq n_r - (j_m/2) \leq (r-2)/2$. 
For any $k\in \ZZ$ satisfying $0 \leq k \leq (r-2)/2$, the value $2\pi k/r$ must lie in Quadrants~I or~II. Then
\begin{equation}\label{Eqn:NR}
[n_r-(j_m/2)]! = \prod_{k=1}^{n_r-(j_m/2)} \frac{\sin(2\pi k/r)}{\sin(2\pi/r)} >0
\end{equation}

By our work checking~(ii), $r/2 < n_r + (j_m/2)+1 \leq r-1$ (note the strict inequality since $j\neq0$ so $j\geq2$ by admissibility). Any $k\in\ZZ$ in this range satisfies $\pi < (2\pi k)/r \leq 2\pi(r-1)/r$, hence it lies in Quadrants~III or~IV and the value of $\sin(2\pi k/r)$ is negative. Then
\begin{equation}\label{Eqn:NRP}
[n_r + (j_m/2)+1]! = \prod_{k=1}^{(r-1)/2} \frac{\sin(2\pi k/r)}{\sin(2\pi/r)} \prod_{(r+1)/2}^{n_r+(j_m/2)+1} \frac{\sin(2\pi k/r)}{\sin(2\pi/r)}
\end{equation}
The first product consists of all positive terms, the second all negative, hence its sign is given by $(-1)^{n_r+(j_m/2)+1-(r-1)/2}$. It follows that the sign of $S_z^{j_1,j_2}$ is given by the product of this sign and $(-1)^{n_r+ (j_m/2)}$, which is $(-1)^{-(r-3)/2}$ since $j_m$ is even. Thus the sign is independent of $z$, $j_1$, $j_2$. It is $1$ if $r\equiv 3\mod4$ and $-1$ if $r\equiv 1\mod 4$. 

Now, for $l\in\{1,2\}$, by \refeqn{DeltaEqn}, 
\[
\Delta(n_r,n_r,j_l)^2=\frac{[n_r-\frac{j_l}{2}]!([\frac{j_l}{2}]!)^2}{[n_r+\frac{j_l}{2}+1]!}
\]
Since $([\frac{j_l}{2}]!)^2$ is positive by \reflem{jbounded}, and $[n_r-j_\ell/2]!$ is positive by \refeqn{NR}, the sign is that of \refeqn{NRP}, or
\[\sgn \left(\Delta(n_r,n_r,j_l)^2 \right) = (-1)^{n_r+(j_l/2)-(r-3)/2}\]
It follows that
\[
\sgn\left((-1)^{(j_1+j_2)/2}\Delta(n_r,n_r,j_1)^2\Delta(n_r,n_r,j_2)^2 \right)=1\]
This concludes the proof in the case $j_1,j_2\neq0$.

Now, suppose one of $j_1$ or $j_2$ is zero.
In this case, consider
$\Tet\begin{bmatrix}n_r & n_r & j_1 \\n_r & n_r & j_2 \end{bmatrix}$
in \reffig{theta-net}~(3). If one of $j_1$, $j_2$ is zero, without loss of generality say $j_1=0$, then there are two 3-vertices in that figure labeled $(n_r, n_r, 0)$. By \refdef{3vertDef}, it follows that labels for $j$, $k$, $c$ in \reffig{theta-net}~(1) are all zero, and labels $a$, $b$, $i$ are all $n_r$. This describes a single edge decorated by the Jones--Wenzl idempotent labeled $n_r$. Thus, we may remove the edge of the tetrahedral network labeled $0$ and replace the adjacent two edges labeled $n_r$ by single edges labeled $n_r$. The result is a theta net labeled $n_r$, $n_r$, $j_2$. Noting the trihedron coefficient is symmetric in its arguments, we conclude
\[\Tet\begin{bmatrix}
n_r & n_r & 0 \\
n_r & n_r & j_2
\end{bmatrix}=\theta(n_r,n_r,j_2)\]
Moreover, a similar argument as above shows that the theta net labeled $(n_r, n_r, 0)$ is just the closure of the Jones--Wenzl idempotent in the plane, so $\theta(n_r,n_r,0)=\Delta_{n_r}$.
Hence, by \refeqn{Tetrato6j},
\[\begin{vmatrix}
n_r & n_r & 0 \\
n_r & n_r & j_2
\end{vmatrix} = \frac{1}{\Delta_{n_r}\theta(n_r,n_r,j_2)}
\Tet\begin{bmatrix}
n_r & n_r & 0 \\
n_r & n_r & j_2
\end{bmatrix} = \Delta_{n_r}^{-1} \]

Applying \reflem{DeltaLambdaSigns}~(1), 
\[ \sgn\left(\begin{vmatrix}
  n_r & n_r & 0 \\
  n_r & n_r & j_2
\end{vmatrix}\right)= \sgn(\Delta_{n_r}) = 
\begin{cases}
  1, & r\equiv3\mod4, \\
  -1, & r\equiv1\mod4.
\end{cases} \]

Finally, suppose $j_1=j_2=0$. Then by \refeqn{Tetrato6j},
\[\begin{vmatrix}
n_r & n_r & 0 \\
n_r & n_r & 0
\end{vmatrix} = \Delta_{n_r}^{-2} \Tet\begin{bmatrix}
n_r & n_r & 0 \\
n_r & n_r & 0
\end{bmatrix} = \Delta_{n_r}^{-2}\Delta_{n_r}=\Delta_{n_r}^{-1} \]
and the result follows as in the previous case.
\end{proof}

\begin{corollary} \label{Cor:SignofN}
Let $L$ be a flat octahedral fully augmented link with $c$ crossing circles, and let $N_{\pmb{i},j_1,j_2,\dots,j_c}$ be as in \refprop{ColJonesforOctFAL}. Let $\pmb{n}_r=(n_r,\dots,n_r)$. Then $N_{\pmb{n}_r,j_1,j_2,\dots,j_c}$ is real-valued, with sign given by
\[\sgn(N_{\pmb{n}_r,j_1,j_2,\dots,j_c})=
\begin{cases}
  1, & r\equiv3\mod4, \\
  -1, & r\equiv1\mod4.
\end{cases}\]
In particular, the sign of $N_{\pmb{n}_r,j_1,j_2,\dots,j_n}$ is independent of $j_1,j_2,\dots,j_c$.
\end{corollary}

\begin{proof}
The fact that $N_{\pmb{n}_r,j_1,\dots,j_n}$ is real-valued follows from the fact that it is a product of real-valued terms by \reflem{DeltaLambdaSigns} and \reflem{6jSign}. Moreover, if $r\equiv3\mod4$ then $\sgn(N_{\pmb{n}_r,j_1,\dots,j_n})=1$ and if $r\equiv1\mod4$ then $\sgn(N_{\pmb{n}_r,j_1,j_2,\dots,j_n})=(-1)^{c}(-1)^{c-1}=-1$, where the term $(-1)^c$ is the sign of the product of the $c$ terms $\Delta_{j_l}\lambda_{j_l,n_r}$ by \reflem{DeltaLambdaSigns}, and the $(-1)^{c-1}$ comes from the $c-1$ quantum 6j-symbols, by \reflem{6jSign}.
\end{proof}

\subsection{Bounds}
Belletti, Detcherry, Kalfagianni and Yang~\cite{MainPaper} give the following upper bound on the growth of rate of the quantum $6j$-symbol. Related results on these growth rates are also due to Costantino~\cite{Cost6jSymb}. 

\begin{theorem}[Theorem~1.2 of \cite{MainPaper}] \label{Thm:6jBound}
For any $r$ and any $r$-admissible 6-tuple $(i,j,k,l,m,n)$, the following bound holds:
\[\frac{2\pi}{r}\log\left|\begin{vmatrix}
i & j & k \\
l & m & n
\end{vmatrix}_{q=\qroot}\right|\leq v_8+\mathcal{O}\left(\frac{\log(r)}{r}\right)\]
\end{theorem}

For a specific choice of colours, the following result is given by Kumar and Melby~\cite{KumarMelbyTVadditivity}; compare also to \cite[Lemma~3.13]{MainPaper}.

\begin{theorem}[Lemma~3.6 of \cite{KumarMelbyTVadditivity}] \label{Thm:lowerBoundChoice}
If the sign is chosen such that $\frac{r-2\pm1}{2}$ is even, then the following holds:
\[\frac{2\pi}{r}\log\Bigg|\begin{vmatrix}
\frac{r-2\pm1}{2} & \frac{r-2\pm1}{2} & \frac{r-2\pm1}{2} \\[2pt]
\frac{r-2\pm1}{2} & \frac{r-2\pm1}{2} & \frac{r-2\pm1}{2}
\end{vmatrix}_{q=\qroot}\Bigg|=v_8+\mathcal{O}\left(\frac{\log(r)}{r}\right)\]
\end{theorem}

\begin{lemma} \label{Lem:DeltaLambdaBound}
For any fixed integers $j$ and $a$ with $0\leq j,a<r-1$,
\[|\Delta_j\lambda_{j,a}|_{q=\qroot}=\mathcal{O}\left(r\right)\]
\end{lemma}

\begin{proof}
For $q=A^2=\qroot$, by \refeqn{LambdaFormula} and \refeqn{Delta_n} we have
\[
\Delta_j\lambda_{j,a}=(-1)^{a+j} \, \frac{\sin(2(j+1)(a+1)\pi/r)}{\sin(2\pi/r)}
\]
Thus, $|\Delta_j\lambda_{j,a}|\leq|\sin(2\pi/r)|^{-1}\leq r/(2 \pi)$ for sufficiently large $r$.
\end{proof}

\begin{proposition}\label{Prop:UpperBound}
Let $L$ be an octahedral fully augmented link with $c$ crossing circles. Then
\[
\frac{2\pi}{r}\log\left|TV_r(S^3\setminus L,q)\right|\leq2(c-1)v_8+\mathcal{O}\left(\frac{\log(r)}{r}\right)
\]
\end{proposition}

\begin{proof}
Let $s$ denote the number of knot strands of $L$. 
By \refthm{TVfromColJones} and \refdef{CJP}
\[
TV_r(S^3\setminus L,q) =
2^{c+s-1} (\eta_r')^2 \sum_{1\leq\pmb{i}\leq m} |J_{L,\pmb{i}}(A)|^2 =
  2^{c+s-1}(\eta_r')^2\sum_{1\leq\pmb{i}\leq m} \left| \sum_{j_1,j_2,\dots,j_c} N_{\pmb{i},j_1,j_2,\dots,j_c} \right|^2
\]
where $N_{\pmb{i},j_1,j_2,\dots,j_c}$ is as in \refprop{ColJonesforOctFAL}.

Let $C_r$ be the number of possible colours $\pmb{i}$ such that $1\leq\pmb{i}\leq m  =(r-1)/2$. Then $C_r\leq m^{c+s}$ grows at most polynomially in $r$. The term $\eta_r' = (q-q^{-1})/\sqrt{-r}$ also grows at most polynomially in $r$.
Observe that $|J_{L,\pmb{i}}(A)|$ consists of $c$ nested sums. By admissibility, each summation variable $j_k$ is such that $0\leq j_k\leq r-2$ for $1\leq k\leq c$. Hence, $|J_{L,\pmb{i}}(A)|$ consists of at most $(r-1)^c$ terms.

Let $\pmb{I}_r$ denote the index realising $\max_{1\leq\pmb{i}\leq m}|J_{L,\pmb{i}}(A)|$, so $|J_{L,\pmb{I}_r}(A)|$ is this maximum value. 
Let $J = \{j_1, \dots, j_n\}$ be the choice of summation variables that realise the maximum of $|N_{\pmb{I}_r,j_1, \dots, j_n}|$. So $|N_{\pmb{I}_r,J}|$ is this maximum. 
Then we obtain the following bounds.
\[
\frac{2\pi}{r}\log\left|TV_r(S^3\setminus L,q)\right| \leq \frac{2\pi}{r}\log|J_{L,\pmb{I}_r}(A)|^2 + \mathcal{O}\left( \frac{\log(r)}{r}\right) \leq \frac{2\pi}{r} \cdot 2\log|N_{\pmb{I}_r,J}|+\mathcal{O}\left(\frac{\log(r)}{r}\right)
\]

Now, $N_{\pmb{I}_r, J}$ is a product of terms including $c-1$ quantum $6j$-symbols.
Thus applying \reflem{DeltaLambdaBound} and \refthm{6jBound} we obtain 
\[
\frac{2\pi}{r}\cdot 2\log|N_{\pmb{I}_r,J}|
\leq 2(c-1)v_8 + \mathcal{O}\left(\frac{\log(r)}{r}\right)  \qedhere\]
\end{proof}

\begin{lemma}\label{Lem:LowerBound}
Let $L$ be an octahedral fully augmented link with $c$ crossing circles, and assume that $L$ has no half-twists. Then for odd $r$, 
\[ \lim_{r\to\infty}\frac{2\pi}{r}\log\left|TV_r(S^3\setminus L,q)\right| \geq  2(c-1)v_8\]
\end{lemma}

\begin{proof}
Let $s$ denote the number of knot strands of $L$, so its total number of components is $c+s$. Consider the colouring $\pmb{n}_r=(n_r,\dots,n_r)$. Using \refthm{TVfromColJones} and the fact that $\eta_r'$ grows at most polynomially in $r$, 
\begin{align*}
  \lim_{r\to\infty}\frac{2\pi}{r}\log\left|TV_r(S^3\setminus L,q)\right| &\ge \lim_{r\to\infty}\frac{2\pi}{r}\log|2^{c+s-1}(\eta_r')^2J_{L,\pmb{n}_r}(A)|^2 \\
& = 2\lim_{r\to\infty}\frac{2\pi}{r}\log|J_{L,\pmb{n}_r}(A)| = 2\lim_{r\to\infty}\frac{2\pi}{r}\log\left|\sum_{j_1,\dots,j_c}N_{\pmb{n}_r,j_1,\dots,j_c}\right|
\end{align*}

Since $L$ has no half-twists, by \refcor{SignofN}, all summands in $|J_{L,\pmb{n}_r}(A)|$ have the same sign, so we can bound the sum below by an individual summand. In particular, we can bound $\left| \sum_{j_1, \dots, j_c} N_{\pmb{n}_r,j_1, \dots, j_c} \right|$
from below by the term $|N_{\pmb{n}_r, n_r, \dots, n_r}|$. That is, take the term with $j_1 = j_2 = \dots = j_c = n_r$, noting that this is one of the terms in the sum, since the sum runs over all $j_k$ with $(n_r, n_r, j_k)$ admissible and $0 \leq j_k \leq r-3$, for $1 \leq k \leq c$. The triple $(n_r,n_r,n_r)$ is such an admissible triple. Hence,
\[
\lim_{r\to\infty}\frac{2\pi}{r}\log\left|TV_r(S^3\setminus L,q)\right| \geq \lim_{r\to\infty}\frac{2\pi}{r}\log|N_{\pmb{n}_r,n_r,\dots,n_r}| = 2(c-1)v_8
\]
where the final equality follows from \reflem{DeltaLambdaBound} and \refthm{lowerBoundChoice}.
\end{proof}

We are now ready to conclude our new proof of the TV volume conjecture for octahedral fully augmented links, a result which is originally due to \cite{WongYang,MainPaper}. 

\begin{theorem} \label{Thm:TVconjforOctFAL}
Let $L$ be an octahedral fully augmented link with $c$ crossing circles and no half-twists. Then for odd $r$,
\begin{equation}\label{Eqn:TVConj}
\lim_{r \to \infty}\frac{2\pi}{r}\log|TV_r(S^3\setminus L,q=\qroot)| = \vol(S^3\setminus L)
\end{equation}
\end{theorem}

\begin{proof}
By \refprop{UpperBound}, the limit on the left hand side of \refeqn{TVConj} is bounded above by $2(c-1)v_8$. 

By \reflem{LowerBound}, the left hand side of \refeqn{TVConj} is also bounded below by $2(c-1)v_8$.

Thus the limit equals $2(c-1)v_8$, which is the volume of $S^3\setminus L$ by \refprop{octaLinksCharac}.
\end{proof}

In \cite{DKYRTTV}, Detcherry, Kalfagianni, and Yang proposed a question about the asymptotic behaviour of the coloured Jones polynomial. They asked whether $J_{L, m}(t)$ for $q^2= t = e^{(2 \pi i)/(m+ \frac{1}{2})}$ grows exponentially in $m$ with growth rate equal to the hyperbolic volume. We answer their question in the positive for flat fully augmented links when $m$ varies over the even integers.  

\begin{theorem}\label{Thm:coloredJonesQuestion}
Let $L$ be an octahedral fully augmented link with $c$ crossing circles and no half-twists. Then as $m$ varies over th even integers, 
\[
\lim_{m \to \infty} \frac{4 \pi}{2m+1} \log \left|J_{L, (m,m\dots,m)} \left(t=e^{\frac{4 \pi \sqrt{-1}}{2m+1}} \right) \right|=2(c-1)v_8 = \vol(S^3 \backslash L),
\]
\end{theorem}

\begin{proof}
Let  $r = 2m+1$ and $n_r = (r-1)/2$. Then as $m$ varies over the even integers
\begin{align*} 
\lim_{m \to \infty} \frac{4 \pi}{2m+1} \log 
\left|J_{L, (m,\dots,m)} \left(t=e^{\frac{4 \pi \sqrt{-1}}{2m+1}} \right) \right| &= 2 \lim_{r \to \infty } \frac{2 \pi}{r} \log \left|J_{L, \pmb{n}_r} \left(q=e^{\frac{2 \pi \sqrt{-1}}{r}} \right) \right| \\
&= 2\lim_{r \to \infty}\frac{2\pi}{r}\log\left|\sum_{j_1,j_2,\dots,j_c}N_{\pmb{n}_r,j_1,j_2,\dots,j_c}\right|
\end{align*}
as $r$ varies over the odd integers.
The result follows as in the proof of \reflem{LowerBound}.
\end{proof}

\begin{remark}
We believe that \refthm{coloredJonesQuestion} can be extended to all $m$ by minor modifications to \reflem{DeltaLambdaSigns} and \reflem{6jSign}, and by using Theorem~A.1 of \cite{MainPaper} to show that
\[ \lim_{r \to \infty} \frac{2 \pi}{r} \log \begin{vmatrix}
\frac{r - 1}{2} & \frac{r - 1}{2} & \frac{r-2\pm 1}{2} \\
\frac{r - 1}{2} & \frac{r - 1}{2} & \frac{r-2\pm 1}{2}
\end{vmatrix} = v_8.\]
\end{remark}

We now give a brief discussion of the issues which we encounter when we attempt to extend \refthm{TVconjforOctFAL} to octahedral fully augmented links with half-twists. First note the following lemma.

\begin{lemma} \label{Lem:gammaBound}
For a $r$-admissible triple $(a,b,c)$,
    \[\lim_{r\to\infty}\frac{2\pi}{r}\log(|\gamma_c^{ab}|\Big|_{q=\qroot})=0.\]
\end{lemma}

\begin{proof} 
    The formula for $\gamma_c^{ab}$ is given in \reflem{removeHalfTwist}. But $A\in\CC$ with $|A|=1$, so it follows that $|\gamma_c^{ab}|=1$.  
\end{proof}

By applying \reflem{gammaBound}, we can prove an upper bound of $2(n-1)v_8$ for octahedral fully augmented links with $c$ crossing circles and allowing for half-twists, as in \refthm{TVconjforOctFAL}. The difficulty is in proving the lower bound. By \refprop{ColJonesforOctFAL}, each half-twist contributes a term $\gamma^{i,k}_j$, where $i,k$ are colours of knot strands and $j$ is a summation variable. For the colouring $\pmb{n}_{r}$, we have $j\in\{0,2,4,\dots,r-3\}$, so $\gamma^{n_r,n_r}_j$ changes sign. Hence, the product of all $\gamma^{n_r,n_r}_j$'s from the half-twists changes sign depending on the summation variables. Moreover, $\gamma^{i,k}_j$ can be complex valued. It is not clear to us how to show that the term $N_{\pmb{n}_r,n_r,\dots,n_r}$ does not cancel with other terms, or to show that this term is a lower bound of the sum in \refeqn{colJonesFormula}. One may, however, be able to use methods from analysis as in \cite{Ohtsuki52knot} to prove a lower bound.

Note that such issues of sign do not arise in Belletti, Detcherry, Kalfagianni and Yang's proof~\cite{MainPaper} of \refconj{TVconj} for fundamental shadow links, as they extend a result of~\cite{DKYRTTV} to prove a result relating the Turaev--Viro invariants to relative Reshetikhin-Turaev invariants \cite{BHMVtqft, LickorishSkein}. The difference in the way these invariants are computed as opposed to the coloured Jones polynomials means that the only contributions from half-twists are products of the modulus $|\gamma^{i,k}_j|$. Since $|\gamma^{i,k}_j|=1$, these contributions do not need to be considered.

\bibliographystyle{amsplain}
\bibliography{ref1}

\end{document}